\documentclass{amsart}

\usepackage{amsmath}
\usepackage{amsthm}
\usepackage{amsfonts}
\usepackage{amssymb}
\usepackage{latexsym}
\usepackage{float}
\usepackage{upref}
\restylefloat{figure}
\usepackage{mathptmx}
\usepackage{graphicx}
\usepackage{color}

\def\c{c}
\def\u{{\mathbf u}}
\def\v{{\mathbf v}}
\def\x{{\mathbf x}}
\def\y{{\mathbf y}}
\def\z{{\mathbf z}}
\def\grad{\nabla}
\def\div{div\,}

\def\RR{\mathbb R}

\def\I{V_0}
\def\J{d_0}
\def\K{B_0}
\def\ThU{\sigma_+}
\def\ThL{\sigma_-}
\def\ThUL{\sigma_\pm}
\def\Vx{S}
\def\Vu{V}
\newcommand{\Lip}[1]{\|#1\|_{\dot{W}^{1,\infty}}}

\newtheorem{thm}{Theorem}[section]
\newtheorem{lem}[thm]{Lemma}
\newtheorem{cor}[thm]{Corollary}
\newtheorem{prop}[thm]{Proposition}
\theoremstyle{definition}
\newtheorem{defn}{Definition}[section]
\newtheorem{rem}{Remark}[section]

\numberwithin{equation}{section}
\numberwithin{thm}{section}
\numberwithin{rem}{section}
\numberwithin{defn}{section}

\begin{document}

\title[Critical thresholds in flocking hydrodynamics with
  nonlocal alignment]{Critical thresholds in flocking hydrodynamics\\with
  nonlocal alignment}

\author[Eitan Tadmor and Changhui Tan]{
Eitan Tadmor$^{1,2,3}$ and Changhui Tan$^{1,2}$}

\address{\mbox{}\newline
$^{1}$Center for Scientific Computation and Mathematical  Modeling\newline
$^{2}$Department of Mathematics\newline
$^{3}$Institute for Physical Science \& Technology\newline \ University of Maryland, College Park}

\subjclass{92D25, 35Q35, 76N10}

\keywords{flocking, alignment, hydrodynamics, regularity, critical thresholds.}


\begin{abstract}
We study the large-time behavior of Eulerian systems augmented with non-local alignment. Such systems arise as hydrodynamic descriptions of agent-based models for self-organized dynamics, e.g., Cucker-Smale and Motsch-Tadmor  models \cite{CS,MT}. We prove that in analogy with the agent-based models, the presence of non-local alignment enforces \emph{strong} solutions to self-organize into a macroscopic flock. This then raises the question of existence of such strong solutions. We address this question in one- and two-dimensional setups, proving global regularity for \emph{sub-critical} initial data. Indeed, we show that there exist \emph{critical thresholds} in the phase space of initial configuration which dictate the global regularity vs. a finite time blow-up. In particular, we explore the regularity of nonlocal alignment in the presence of vacuum.
\end{abstract}

\maketitle


\section{Introduction}
We consider the following system of hydrodynamic equations for density $\rho(\x,t)$ and velocity $\u(\x,t)$
\begin{subequations}\label{eq:first}
\begin{align}
\rho_t+\div(\rho\u)&=0,\qquad \x\in\RR^n, t\geq 0,\\
\u_t+\u\cdot\grad\u+\grad p&
=\int_{\RR^n}a(\x,\y)(\u(\y)-\u(\x))\rho(\y)d\y,\label{eq:firstu}
\end{align}
subject to initial conditions
\begin{equation}
\rho(\x,0)=\rho_0(\x)\geq 0,\qquad \u(\x,0)=\u_0(\x).
\end{equation}
\end{subequations}
Such systems arise as macroscopic description of 
the agent based models, e.g. \cite{HT}, in which every agent adjusts its velocity to that of its neighbors through the process of \emph{alignment}, dictated by the \emph{interaction kernel} $a(\cdot,\cdot)$, 
\begin{equation}\label{eq:particle}
\displaystyle\dot{\x}_i=\v_i,\quad
\dot{\v}_i=\sum_{j=1}^N a(\x_i,\x_j)(\v_j-\v_i).
\end{equation}


\noindent
The expected long time behavior of these agents is to self-organize into finitely many \emph{clusters}, and in particular, depending on the properties of the {interaction kernel}, $a(\cdot,\cdot)$, to flock into one such cluster;  consult the recent reviews \cite{VZ,MT2}. The goal of this paper is to study the \emph{flocking} phenomenon of the corresponding hydrodynamic description (\ref{eq:first}), or flocking hydrodynamics for short.

We shall discuss two prototype models: the celebrated model
of Cucker-Smale (CS) \cite{CS,CS2} which employs a symmetric
interaction kernel
\begin{subequations}\label{eqs:mod}
\begin{equation}\label{eq:modCS}
a(\x,\y)=\phi(|\x-\y|),
\end{equation}
and a related model of Motsch \& Tadmor (MT) introduced in \cite{MT},
which provides a better description of far-from-equilibrium
 flocking dynamics, using a non-symmetric (time-dependent) interaction of the form,
\begin{equation}\label{eq:modMT}
a(\x,\y)=\frac{\phi(|\x-\y|)}{\int_{\RR^n}\phi(|\x-\z|)\rho(\z,t)d\z},
\end{equation}
\end{subequations}
Here, $\phi=\phi(r)$ is the \emph{influence function} which is  assumed
to decrease in $r$, thus reflecting the intuition that alignment
becomes weaker as the distance becomes larger (but consult \cite{MT2}).

To put our discussion into context, consider the hydrodynamic model (\ref{eq:first}) with the symmetric CS alignment $a(\x,\y)=\phi(|\x-\y|)$. There are
three different regimes of interest, depending on the behavior of $\phi$  near the origin and at infinity ---
local dissipation, fractional dissipation and nonlocal alignment. Here
is a brief overview.

\textit{\#1. Local dissipation.} Assume $\phi$ is bounded and decay
sufficiently fast at infinity, such that
$\displaystyle\int_0^\infty\phi(r)r^{n+1}dr$ is finite (in particular, including compactly supported $\phi$'s).
We process the hyperbolic scaling, $(\x,t)\mapsto\left(\frac{\x}{\epsilon},\frac{t}{\epsilon}\right)$ and $\phi\mapsto\phi_\epsilon:=
\frac{1}{\epsilon^{n+2}}\phi\left(\frac{|\x|}{\epsilon}\right), \ \rho \mapsto \frac{\rho}{\epsilon}$,
and let $\epsilon$ go to zero: one arrives at the following system (expressed in the usual conservative form in terms  $\omega_n$,  the surface area of the unit sphere in $\RR^n$)
\begin{subequations}\label{eq:local}
\begin{align}
&\rho_t+\div(\rho\u)=0,\label{eq:localrho}\\
& (\rho\u)_t+\div(\rho\u\otimes\u)+\grad P=
C\div\left(\mu(\rho)\grad\u\right), \quad C:=\displaystyle\frac{\omega_n}{2n}\int_0^\infty\phi(r)r^{n+1}dr,\label{eq:localm}
\end{align}
\end{subequations}
with pressure $P(\rho):=\int^\rho s p'(s)ds$ and viscosity coefficient $\mu(\rho)=\rho^2$.
System \eqref{eq:localrho}\eqref{eq:localm} belongs to the class of
compressible Navier-Stokes equations
with degenerate viscosity $\mu=\rho^\theta, \theta>0$ which vanishes at the vacuum.
The study of such equations is mostly limited to  one dimension.
For existence and uniqueness of  weak solution with ``moderate degeneracy'', $\theta<1/2$, we refer to
\cite{LXY,YYZ,YZ}.
Mellet and Vasseur \cite{MV} proved that the degenerate viscosity is nevertheless strong enough to enforce global
existence and uniqueness of the strong solution away from vacuum
assuming $\rho_0>0$.

\textit{\#2. Fractional dissipation.}
Consider an influence function with a sufficiently strong singularity at the origin,
$\phi(r)\sim r^{-n+2\alpha}$, associated with fractional dissipation of order $2\alpha$. The corresponding 
incompressible setup (setting, formally, $\rho \equiv 1$ in \eqref{eq:first}) reads
\[
\u_t+\u\cdot\grad\u+\grad p=-(-\Delta)^\alpha\u, \qquad 
\div\u=0.
\]
$L^2$-energy bound implies that global smooth solutions exist for $\alpha>\frac{1}{2}+\frac{n}{4}$, e.g. \cite{KP2,Wu}. Additional pointwise bounds available in the one-dimensional case, imply that Burgers' equation with fractional dissipation, $\u_t+\u\cdot\grad\u=-(-\Delta)^\alpha\u$, admits global solutions  for $\alpha > 1/2$; the critical case, $\alpha=1/2$ was the subject of  extensive recent studies  \cite{CV,KNS,CoV}

\textit{\#3. Nonlocal alignment.}
In this paper we focus our attention on the remaining case where $\phi$ is bounded at the origin and decays sufficiently slow at infinity. As a prototypical example we may consider $\phi(r)=(1+r)^{-\alpha}$ with $\alpha<1$: the nonlocal alignment (due to the divergent tail of $\int^\infty \phi(r)dr$)  enforces an unconditional flocking of both the CS and MT models (\ref{eq:first}),(\ref{eqs:mod}),  consult \cite{CS,HT,HL,MT2}. 

The main system we are concerned with in this paper is the
corresponding ``pressure-less'' compressible Euler equations with nonlocal alignment\footnote{The same phenomenon of flocking hydrodynamics for sub-critical initial data occurs in the presence  of pressure, e.g., \cite{TW}
and will be discussed in a future work.}
\begin{subequations}\label{eq:main}
\begin{align}
\rho_t+\div(\rho\u)&=0,\label{eq:rho}\\
\u_t+\u\cdot\grad\u&
=\int_{\RR^n}\frac{\phi(|\x-\y|)}{\Phi(\x,t)}(\u(\y)-\u(\x))\rho(\y)d\y, 
\qquad \x\in\RR^n, t\geq 0, \label{eq:u}
\end{align}
subject to compactly supported initial density $\rho_0$ and uniformly
bounded initial velocity $\u_0$,
\begin{equation}\label{eq:IC}
\rho(\x,0)=\rho_0(\x)\in L^1_+(\RR^n),\qquad \u(\x,0)=\u_0(\x)\in W^{1,\infty}(\RR^n).
\end{equation}

The  scaling function $\Phi(\x,t)$ corresponds to the CS and MT models,
\begin{equation}
\Phi(\x,t)\left\{\begin{array}{ll} \equiv 1 & \mbox{CS model},\\
                                                \displaystyle =\int_{\RR^n} \phi(|\x-\y)|)\rho(\y,t)d\y, & \mbox{MT model}. \end{array}\right. 
\end{equation}
\end{subequations}
Throughout the paper, the influence function, $\phi$, is assumed non-increasing, Lipschitz and non-local in the sense of having a  divergent  tail
\begin{equation}\label{eq:phi2}\tag{H}
\int^\infty\phi(r)dr=\infty.
\end{equation}
Without loss of generality, we assume $\phi(0)=1$.

Our first main result, stated in  theorem \ref{thm:flock} below shows, in analogy with the agent-based models, that Lipschitz solutions of \eqref{eq:main} driven by nonlocal alignment are self-organized into a macroscopic flock. It is therefore natural to ask, when does the system (\ref{eq:main}) preserve the Lipschitz
regularity of $\u(\cdot,t)$ for all time? This question regarding the uniform global bound   $\|\grad_\x\u(\cdot,t)\|_{L^\infty} < \infty$  occupies the rest of the paper. We begin with the one-dimensional case.
Our second main result, summarized in corollary \ref{thm:global}, shows that there exists a
large set of so-called \emph{sub-critical initial
configurations}, under which $\|u_x(\cdot,t)\|_{L^\infty}$
remains bounded for all time. On the other hand, there exists
a set of super-critical initial configurations which leads to the lose of regularity at finite time,
$\|\grad\u(\cdot,t)\|_{L^\infty} \rightarrow \infty$ as $t\uparrow T_c$. The so called \emph{critical threshold} phenomenon 
 for Eulerian dynamics was first systematically 
studied in \cite{ELT} in Euler-Poisson equations, followed by a series
of related studies \cite{LT2,LT3,CT,TW}.
In particular, Liu and Tadmor \cite{LT}  studied the critical threshold
phenomenon of the one-dimensional Burgers equation with nonlocal
convolution source term,
$ \displaystyle u_t+uu_x=\int_{-\infty}^\infty\phi(|x-y|)(u(y)-u(x))dy$,
 corresponding to  the one-dimensional
CS alignment system \eqref{eq:u} with $\rho\equiv1$.
Here, we extend this result, proving critical threshold phenomenon for the non-vacuum density-dependent model (theorem \ref{thm:1DNF}). Thus, global regularity and hence flocking follow for sub-critical data, which in turn enables us to 
 significantly improve the critical threshold derived in \cite{LT}.
Next, we  extend the global regularity and hence flocking result of CS hydrodynamics to two-space dimension (theorem \ref{thm:2DCS}). These global regularity results are complemented by the unique feature of the non-local alignment, which prevents finite-time blow-up   dynamics within regions of vacuum  (theorem \ref{thm:vacuum}).
Finally, the flocking of one- and  two-dimensional MT hydrodynamics  is summarized in theorem \ref{thm:mainMT}.

The  paper is organized as follows. We state the main
results in section \ref{sec:results}. In section \ref{sec:flock}, we
prove that global strong solution implies flocking. In section
\ref{sec:nonvacuum}, we prove critical thresholds for Cucker-Smale
system \eqref{eq:main}, in dimension one and two. 
The fast alignment enhances the dynamics and leads to improved critical thresholds; this is carried out in section \ref{sec:fast}. In section
\ref{sec:vacuum}, we discuss the boundedness of $\grad_\x\u$ inside
the vacuum. We conclude in section \ref{sec:MT}, where we extend these results to the Motsch-Tadmor system. 

\section{Statement of main results}\label{sec:results}

\subsection{Flocking of strong solutions}
We begin with the notion of flocking, e.g. \cite{MT}.
\begin{defn}[Flock]\label{def:flock}
We say that a solution $(\rho,\u)$ of (\ref{eq:main}) converges to a flock if the following hold:
\begin{enumerate}
\item Spatial diameter is uniformly bounded, \textit{i.e.}
there exists $D<+\infty$ such that
\[
\Vx(t):=\sup\{|\x-\y|,~\x,\y\in supp(\rho(t))\} \leq D, \quad \forall t>0.
\]
\item Velocity diameter decays to $0$ for large time, \textit{i.e.}
\[
\lim_{t\to\infty}\Vu(t)=0, \quad 
\Vu(t):=\sup\{|\u(\x,t)-\u(\y,t)|,~\x,\y\in supp(\rho(t))\}.
\]
\end{enumerate}
If $V(t)$ decays exponentially fast in time, we say the flock has
\emph{fast alignment} property.
\end{defn}

The following theorem shows flocking property for \emph{strong} (Lipschitz) solutions of the 
non-local alignment  system \eqref{eq:main}. In fact  the following
 fast alignment holds for strong solutions of both CS and MT models.
This is quantified in terms of the following interaction bound, which will be used throughout the paper,
\[
\int_\y a(\x,\y)\rho(\y)d\y  \left\{\begin{array}{ll} \leq m, & \mbox{CS model},\\ \\ \equiv 1, & \mbox{MT model}.\end{array}\right.
\]
Recalling that $\phi(\cdot)\leq 1$ one may use the CS interaction bound $m:=\int_{\RR^n}\rho_0(\y)d\y$, whereas $m=1$ for the MT model.

\begin{thm}[Flock with fast alignment]\label{thm:flock}
Let $(\rho,\u)$ be a global strong solution of system \eqref{eq:main} subject to a compactly supported density $\rho_0=\rho(\cdot,0)$ and bounded velocity $\u_0=\u(\cdot,0)\in L^\infty$. Assume a 
 influence function $\phi$ is global in the sense that 
\begin{equation}\label{eq:phi1}
m\int_{\Vx_0}^\infty\phi(r)dr>\Vu_0.
\end{equation}
Then, $(\rho,\u)$ converges to a flock with fast alignment; specifically, 
there exists a finite number $D$, such that
\[
\sup_{t\geq 0}S(t)\leq D,\quad V(t)\leq V_0 e^{-m\phi(D)t}.
\]
\end{thm}

\begin{rem}
With the slow decay assumption \eqref{eq:phi2} on $\phi$, condition
\eqref{eq:phi1} automatically holds with finite $\Vx_0=\Vx(0)$ and $\Vu_0=\Vu(0)$.  The
constant $D$ depends on $\phi, S_0$ and $V_0$. An explicit expression
of $D$ is given in \eqref{eq:D} below.
\end{rem}
\noindent
It is well-known that strong
solutions persist as long as $\|\grad_\x\u(\cdot,t)\|_{L^\infty}$ remains bounded. Motivated by theorem \ref{thm:flock}, we  study below the set of  initial configuration which guarantee the uniform boundedness of  
$\grad_\x\u$ globally in time, which in turn implies the emergence of a flock.

\subsection{Critical thresholds in one dimensional Cucker-Smale model}
We study the uniform boundedness of $u_x$ for the one dimensional
Cucker-Smale alignment system
\begin{subequations}\label{eq:main1D}
\begin{align}
\rho_t+(\rho u)_x&=0, \label{eq:main1Da}\\
u_t+uu_x&=\int_\RR\phi(|x-y|)(u(y)-u(x))\rho(y)dy,\quad x\in\RR, t\geq0, \label{eq:main1Db}
\end{align}
subject to initial conditions \eqref{eq:IC}, with a non-local interaction 
\eqref{eq:phi2}.
\end{subequations}
In sections \ref{sec:nonvacuum} and \ref{sec:fast} we prove that
 if the initial velocity has a bounded diameter $\I$, and if  its  slope is not too negative relative to $\I$,  then
$\|u_x\|_{L^\infty(supp(\rho))}$ remains uniformly bounded for all time\footnote{Observe that if $\phi_0\equiv 0$ then (\ref{eq:main1Db}) is reduced to the inviscid Burgers' equation with generic finite-time blow-up  unless $d_0>0$. Thus, the addition of non-local alignment has a regularization effect, by increasing the initial threshold for global regularity.}.

\begin{thm}[1D critical thresholds for non-vacuum]\label{thm:1DCS}
Consider initial value problem of \eqref{eq:main1D}. There exists
threshold functions $\ThU>\ThL$ (depending on $\phi$), such that the following hold.\newline
\begin{itemize}
\item If the initial condition satisfies
\begin{equation}\label{eq:ThU}
d_0:=\inf_{x\in supp(\rho_0)}u_{0x}(x)>\ThU(\I), \qquad
\I:=\sup_{x,y\in supp(\rho_0)}|u_0(x)-u_0(y)|.
\end{equation}
then $u_x(x,t)$ remains bounded for all $(x,t)\in supp (\rho)$.\newline
\item If the initial condition satisfies $d_0<\ThL(\I)$,
then there exists a finite time blow-up $t=T_c>0$ such that
 $\inf_{x\in supp(\rho(\cdot,t))} u_x(x,t)\to-\infty$ as $t\rightarrow T_c-$.
\end{itemize}
\end{thm}

\begin{rem}
Detailed expressions of threshold functions $\ThU$ and $\ThL$ are given in
section \ref{sec:1DFF}. Figure \ref{fig:oned} illustrates the two
thresholds. To ensure boundedness of $u_x$, there are two
requirements for the initial configurations:

\qquad (i) The initial slope of velocity $u_{0x}$ is not too negative; and

\qquad (ii) the initial diameter of velocity $V_0$ is not too large.

Note that due to symmetry, the steady state of the CS system (\ref{eq:main1D}) is given by the average value, $u=\bar{u}_0$, and the upper threshold condition (\ref{eq:ThU}) 
tells us  that if the initial configuration is not far away from that equilibrium,
then strong solution exists  and non-local alignment enforces its flocking towards  steady state.
\end{rem}

\begin{figure}[H]
\vspace*{-4.5cm}
\centering
\qquad \qquad \includegraphics[width=.6\linewidth]{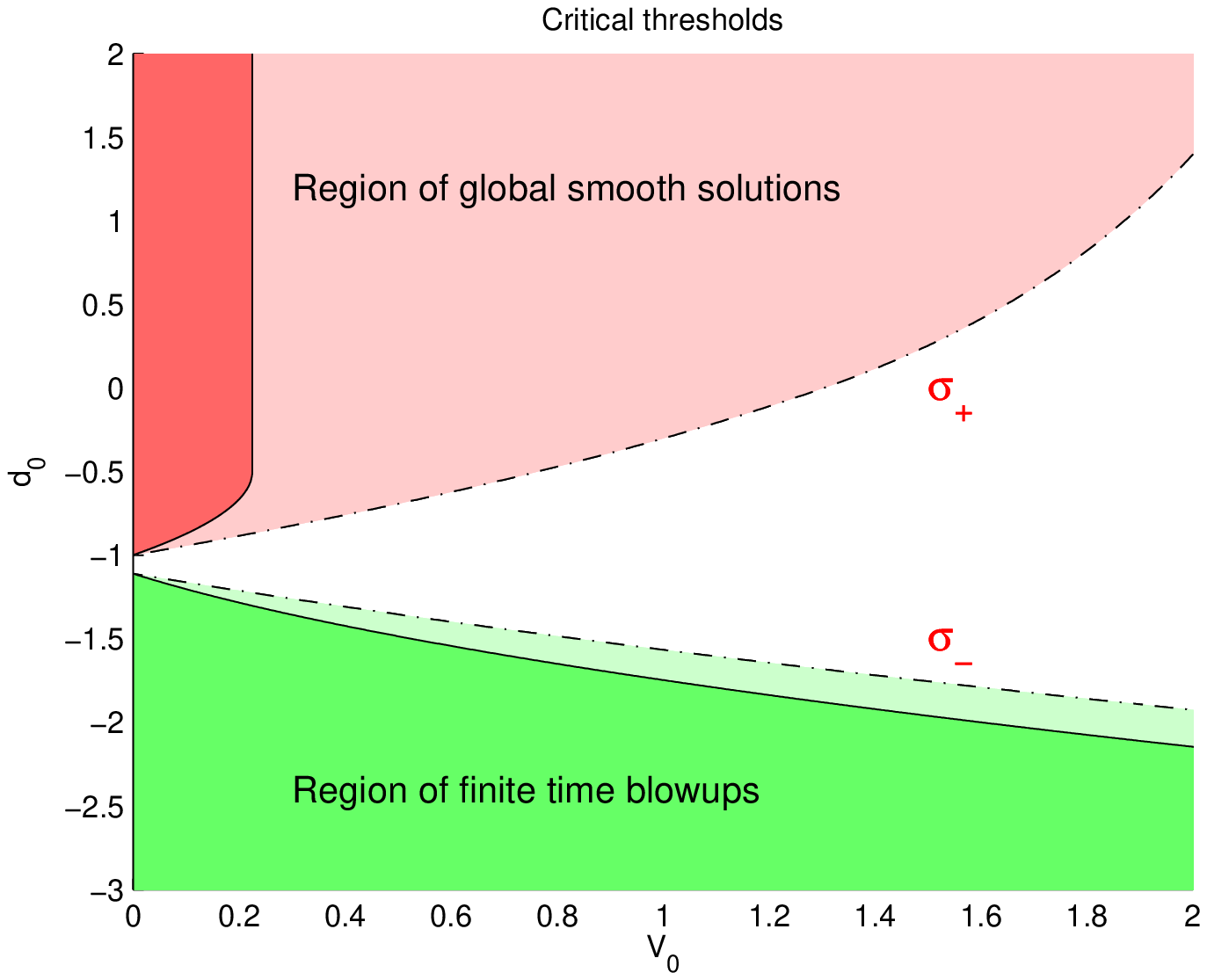}
\vspace*{4.5cm}
\caption{Illustration of the critical thresholds in one dimension}\label{fig:oned}
\end{figure}

\begin{rem}
The darker areas in figure \ref{fig:oned} represents the thresholds
result stated in theorem \ref{thm:1DNF}. It is an extension of
the result in \cite{LT} for the case where $\rho\equiv1$. Taking
advantage of  the fast alignment property, we are able to improve the
result to the lighter area.
\end{rem}

The last theorem is restricted to the non-vacuum portion of the solution.
For local systems, e.g.\eqref{eq:local}, the dynamics  inside the vacuum
acts the same way as the inviscid Burgers equation, with generic formation of  shock discontinuities. It is here that we take advantage of the non-local character of the alignment model (\ref{eq:main}): 
our next theorem identifies an upper threshold which ensures that 
$u_x$ remains bounded even outside the support of $\rho$. Thus, the nonlocal interaction helps smoothing the equation and
enables us to find thresholds in the vacuum region.

\begin{thm}[1D upper threshold for the vacuum region]\label{thm:vacuum}
Consider initial value problem of \eqref{eq:main1D}.
Let $V_0^\lambda$ denote the diameter of the initial velocity
between a point in the non-vacuum region  and a point at most $\lambda$
away from that region,
$V_0^\lambda:=\sup\big\{|u_0(x)-u_0(y)|~:~dist(x,supp(\rho_0))\leq
\lambda,~~ y\in supp(\rho_0),
\big\}$.\newline
If the initial configuration satisfies
\begin{subequations}\label{eq:vacuumCond}
\begin{align}
V_0^\lambda\leq&\frac{m\phi^2(\lambda+D)}{4|\phi'(\lambda)|+2|\phi'(\lambda+D)|}
\quad\text{for all}~\lambda\geq0,\label{eq:vacuumphi}\\
u_{0x}(x)\geq&-\frac{m}{2}\phi( dist(x,supp(\rho_0))+D),
\end{align}
\end{subequations}
then $u_x(x,t)$ remains bounded for all $(x,t)\not\in supp (\rho)$.
\end{thm}
\begin{rem}
Condition \eqref{eq:vacuumCond} has the same flavor as \eqref{eq:ThU}
for the non-vacuum area: the diameter  of the initial velocity is not too large
and the slope of the initial velocity is not too negative. For \eqref{eq:vacuumphi},
when $\lambda$ approaches zero,  the condition is equivalent to the
non-vacuum case. On the other hand, when $\lambda$ approaches
infinity, if $\phi(r)\sim r^{-\alpha}$, the right hand side is
proportional to $r^{1-\alpha}$. Thanks to the slow decay assumption on
$\phi$, \emph{i.e.} $\alpha<1$, \eqref{eq:vacuumphi} provides no
restrictions on $V_0^\infty$. Note that if $\alpha>1$, the condition
requires $V_0^\infty=0$ which can not be achieved unless $u$ is a constant.
\end{rem}
Combining the last two theorems, we conclude that the one-dimensional
CS hydrodynamics \eqref{eq:main1D} has global strong solutions for a  suitable set of \emph{sub-critical initial conditions}.

\begin{cor}[1D global strong solution]\label{thm:global}
Consider the one-dimensional CS system \eqref{eq:main1D} then there exist thresholds $\sigma_\pm$ such that the following holds.\newline
If initial configuration satisfies both \eqref{eq:ThU}, 
  \eqref{eq:vacuumCond}, then there
exists a strong solution $\rho\in L^\infty([0,+\infty),L^1(\RR))$ and $u\in
L^\infty([0,+\infty), W^{1,\infty}(\RR))$. Moreover, the solution
converges to a flock in the sense of definition \ref{def:flock}.\newline
If initial configuration satisfies $d_0<\sigma_-(V_0)$, then
 the corresponding  solution $(\rho,u)$ will blow-up at a finite time.
\end{cor}

\subsection{Critical thresholds for two dimensional Cucker-Smale model}
We extend our main result to two space dimensions, 
\begin{subequations}\label{eq:main_2DCS}
\begin{align}
\rho_t+\div(\rho\u)&=0,\label{eq:main_2DCSa}\\
\u_t+\u\cdot\grad\u&
=\int_{\RR^2}\phi(|\x-\y|)(\u(\y,t)-\u(\x,t))\rho(\y,t)d\y,\qquad \x\in\RR^2, t\geq 0,\label{eq:main_2DCSb}
\end{align}
\end{subequations}
where critical threshold is captured  in terms of 
 $\I$ and its initial divergence
$\displaystyle\J:=\inf_{\x\in supp(\rho_0)}\div\u_0(\x)$.
The main difficulty here is to control  the remaining  terms in 
$\grad\u$, beyond just $\div\u$ itself.  We measure the size of those  additional terms, setting  
\[\K:=\sup_{x\in
  supp(\rho_0)}\max\left\{2|\partial_{x_1}u_{02}|, 2|\partial_{x_2}u_{01}|,
|\partial_{x_1}u_{01}-\partial_{x_2}u_{02}|\right\},\]
and we prove that if $\K$ is sufficiently small then in fact all  terms of $\grad\u$, except for $\div \u$, remain equally small.

\begin{thm}[2D critical thresholds in non-vacuum region]\label{thm:2DCS}
Consider the two-dimensional CS system \eqref{eq:main_2DCS}.
There exist upper threshold functions $\ThU, \zeta$ such that,
if  $\J>\ThU(\I)$ and $\K<\zeta(\I)$ at $t=0$, 
then $\grad_\x\u(\x,t)$ remains bounded for all $(\x,t)\in supp (\rho)$.\newline
On the other hand, there exists a lower threshold function $\ThL$ such that,
if  $\J<\ThL(\I)$, $|\partial_{x_1}u_{02}|$ and
$|\partial_{x_2}u_{01}|$ are large enough at $t=0$,
then there exists a finite time blow-up at $T_c>0$ where  $\inf_{\x\in supp(\rho(\cdot,t))}\div\u(\x,t)\to-\infty$ as $t\rightarrow T_c-$.
\end{thm}

\begin{rem}
The smallness assumption on $B_0$ guarantees that the terms in $\grad\u_0$ remain small  relative to $\div\u(\cdot,t)$. Put differently, theorem \ref{thm:2DCS} states that if the \emph{vorticity}, $\omega_0$ and the \emph{spectral gap}, $\eta_0$, are
small enough at $t=0$, then they will remain small for all time (the result has the same flavor of the critical threshold in the two-dimensional \emph{restricted} Euler-Poisson equations expressed in terms of $(\omega_0,\eta_0)$,  \cite{LT3}). Consult theorem \ref{thm:2D} for all details. 
\end{rem}

\ifx
\begin{rem}
For higher dimensions, $\grad_\x\u$ consists too many extra terms. The dynamics
become too complicated, and we are not able to control all terms. The
strategy stops working.
\end{rem}
\fi

Theorem \ref{thm:2DCS} is restricted to the non-vacuum part of the dynamics. For  vacuum, it is easy to derive a result analogous to the one-dimensional setup in theorem \ref{thm:vacuum}. We omit the details.

\subsection{Critical thresholds for Motsch-Tadmor system}
We extend the main results to macroscopic Motsch-Tadmor system,  $\x\in\RR^n, t\geq 0$ in $n=1,2$ spatial dimensions,
\begin{subequations}\label{eq:MT}
\begin{align}
\rho_t+\div(\rho \u)&=0,\label{eq:MTrho1D}\\
\u_t+\u \cdot\nabla\u&
=\int_{\RR}\frac{\phi(|\x-\y|)}{\Phi(\x,t)}(u(\y,t)-u(\x,t))\rho(\y,t)d\y,
\quad \Phi(\x,t):=\int_{\RR}\phi(|\x-\y|)\rho(\y,t)d\y,
\label{eq:MTu}
\end{align}
\end{subequations}
subject to the same initial condition \eqref{eq:IC}.
 A main difference from the Cucker-Smale system is the lack of  conservation of momentum. As an example, we prove that the analogue of CS non-vacuum threshold dynamics  in one- and two-dimensions hold for \eqref{eq:MT} (albeit with  a different choice of critical $\sigma_\pm$).

\ifx
\begin{thm}[Flock for Motsch-Tadmor system]\label{thm:flockMT}
Let $(\rho,\u)$ be a global strong solution of system \eqref{eq:MT}.
with non-local alignment \eqref{eq:phi1}.
Then, $(\rho,\u)$ converges to a flock with fast alignment, namely,
there exists a finite number $D$, such that
\[\sup_{t\geq 0}S(t)\leq D,\quad V(t)\leq V_0 e^{-\phi(D)t}.\]
\end{thm}
\fi

\begin{thm}[Critical thresholds for Motsch-Tadmor system]\label{thm:mainMT}
Consider the MT hydrodynamics  \eqref{eq:MT}. There exists
threshold functions $\ThU>\ThL$ and $\zeta$, such that the conclusions 
of theorem \ref{thm:1DCS} (in $n=1$ dimension) and theorem \ref{thm:2DCS} (in $n=2$ dimensions) hold.
\end{thm}

\section{Strong solutions must flock}\label{sec:flock}
In this section, we prove theorem \ref{thm:flock}: any global strong solution
of \eqref{eq:main} converges to a flock with fast alignment.
The key idea following \cite{MT2}, is to measure the decay of $\Vx(t)$
and $\Vu(t)$ dynamically, .
\begin{prop}[Decay estimates towards flocking]\label{prop:SDDI}
A strong solution $(\rho,\u)$ of the CS or MT models \eqref{eq:main} satisfies 
\begin{subequations}\label{eq:flock}
\begin{align}
\frac{d}{dt}\Vx(t)&\leq \Vu(t),\label{eq:dX}\\
\frac{d}{dt}\Vu(t)&\leq -m\phi(\Vx(t))\Vu(t).\label{eq:dV}
\end{align}
\end{subequations}
\end{prop}
\begin{proof}
Consider two characteristics $\dot{X}(t)=\u(X,t)$,
$\dot{Y}(t)=\u(Y,t)$ subject to initial conditions $X(0)=\x$,
$Y(0)=\y$, where $\x,\y\in supp(\rho_0)$.
To simplify the notations, we omit the time variable throughout the proof.
First, we compute
$\displaystyle \frac{d}{dt}|Y-X|^2=2\langle Y-X,\u(Y)-\u(X)\rangle\leq 2\Vx\Vu$.
Taking the supreme of the left hand side for all $\x,\y\in
supp(\rho_0)$, yields the inequality  \eqref{eq:dX}.

Next, define
\[
b(\x,\y):=\frac{1}{m}a(\x,\y)\rho(\y)
+\left(1-\frac{1}{m}\int_{\RR^n}a(\x,\y)\rho(\y)d\y\right)\delta_0(\x-\y),
\quad a(\x,\y)= \frac{\phi(|\x-\y|}{\Phi(\x,t)};
\]
note that since  MT model has the ``correct'' scaling (with $m=1$) then  the second term involving Dirac mass  $\delta_0$ drops out. We leave it to the reader to verify that $b(\cdot,\cdot)$ satisfies the two properties:
\begin{itemize}
\item[(P1)] $\displaystyle\int_{\RR^n}b(\x,\y)
d\y=1$, for all $\x\in supp(\rho)$.
\item[(P2)] $\displaystyle\int_{\RR^n}b(\x,\y)(\u(\y)-\u(\x))d\y=
\frac{1}{m}\int_{\RR^n}a(\x,\y)(\u(\y)-\u(\x))\rho(\y)d\y$.
\end{itemize}

To prove \eqref{eq:dV}, we use  the momentum equation (\ref{eq:u}) 
\[
 \frac{d}{dt}|\u(Y)-\u(X)|^2=2\left\langle\u(Y)-\u(X),
\int_{\RR^n}\left[a(Y,\z)(\u(\z)-\u(Y))-a(X,\z)(\u(\z)-\u(X))\right]
\rho(\z)d\z\right\rangle.
\]
We expand  the second expression on the right in terms of 
$\eta_{X,Y}(\z)=\eta(\z):=\min\{b(X,\z),b(Y,\z)\}$:
\begin{align*}
&\int_{\RR^n}\left[a(Y,\z)(\u(\z)-\u(Y))-a(X,\z)(\u(\z)-\u(X))\right]
\rho(\z)d\z\\
&~\stackrel{(P2)}{=}~m\int_{\RR^n}\left[b(Y,\z)(\u(\z)-\u(Y))-b(X,\z)(\u(\z)-\u(X))\right]
d\z\\
&~\stackrel{(P1)}{=}~m\int_{\RR^n}\left(b(Y,\z)-b(X,\z)\right)\u(\z)
d\z-m(\u(Y)-\u(X))\\
&~ ~=~m\int_{\RR^n}(b(Y,\z)-\eta(\z))\u(\z)d\z
-m\int_{\RR^n}(b(X,\z)-\eta(\z))\u(\z)d\z-m(\u(Y)-\u(X)).
\end{align*}
Set $c(\z):=b(Y,\z)-\eta(\z)\geq0$ and $d(\z):=b(X,\z)-\eta(\z)\geq0$;  we find
\begin{align*}
\frac{d}{dt}|\u(Y)-\u(X)|^2  \leq & 2m\int_{\RR^n} c(\z)d\z \times \max_{\z}\left\langle\u(Y)-\u(X),\u(\z)\right\rangle\\
&-2m\int_{\RR^n} d(\z)d\z \times \min_{\z}\left\langle\u(Y)-\u(X),\u(\z)\right\rangle
 -2m|\u(Y)-\u(X)|^2,
\end{align*}
and since by (P1), $\displaystyle \int c(\z)d\z=\int d(\z)d\z =1-\int \eta(\z)d\z$, we end up with
\begin{align*}
\frac{d}{dt}|\u(Y)-\u(X)|^2\leq 2m\left(1-\int \eta(\z)d\z\right)\max_{\x,\y} |\u(\y)-\u(\x)|^2 -2m|\u(Y)-\u(X)|^2.
\end{align*}
Since the support of $\rho$ is compact, we can take the two maximal
characteristics $Y$ and $X$ which realize the diameter $\Vu=|\u(Y)-\u(X)|$. We conclude the decay estimate
\[
\frac{d}{dt}\Vu^2\leq -2m\left(\int \eta(\z)d\z\right)\Vu^2, \qquad
\eta(\z)=\min\{b(X,\z),b(Y,\z)\}.
\]
At the heart of mater is the decay factor $\int \eta(\z)d\z$:  we compute its lower bound for CS model 
\[
\int_{\RR^n} \min\{b(X,\z),b(Y,\z)\}d\z \geq \frac{1}{m}\phi(S) \int_{\RR^n}\rho(\z,t)d\z= \phi(S), \qquad X,Y\in supp(\rho);
\]
similarly, for the MT model (where $m=1$ and $\Phi(X) \leq \int \rho(\z)d\z$) we have
\[
\int_{\RR^n} \min\{b(X,\z),b(Y,\z)\}d\z \geq 
 \frac{\phi(S)}{\max_{X,Y}\{\Phi(X),\Phi(Y)\}}\int_{\RR^n} \rho(\z)d\z \geq \phi(S),\qquad X,Y\in supp(\rho).
\]
The result (\ref{eq:dV}) follows from the last three bounds.
\end{proof}

Equipped with the decay estimates (\ref{eq:flock}, we use the technique introduced in \cite{HL} to prove the flocking behavior of (\ref{eq:main}). 
\begin{proof}[Proof of Theorem \ref{thm:flock}]
Consider free energy
$\displaystyle \mathcal{E}:=\Vu+m\int_0^{\Vx}\phi(s)ds$. The decay estimates \eqref{eq:flock} imply
$\frac{d}{dt}\mathcal{E}\leq0$ and hence 
$\Vu(t)-\Vu_0\leq-m\int_{\Vx_0}^{\Vx(t)}\phi(s)ds$.
By assumption \eqref{eq:phi1},  there exists a finite number
$D$ (depending  on $\phi, \rho_0, \u_0$), such that
\begin{equation}\label{eq:D}
D:=\psi^{-1}\big(\Vu_0+\psi(\Vx_0)\big),~~
\text{where}~~\psi(t)=m\int_0^t\phi(s)ds,
\end{equation}
for which $\displaystyle \Vu_0=m\int_{\Vx_0}^{D}\phi(s)ds$. Hence, we have
$\displaystyle 0\leq \Vu(t)\leq m\int_{\Vx(t)}^{D}\phi(s)ds$.
In particular, it yields that $\Vx(t)\leq D<\infty$, and since $\phi$ is monotonically decreasing, \eqref{eq:dV} yields
\[\frac{d}{dt}\Vu(t)\leq -m\phi(D)\Vu(t)
\quad\leadsto\quad
\Vu(t)\leq \Vu_0e^{-m\phi(D)t}\to 0, ~~\text{as}~t\to+\infty.\]
We conclude that $(\rho,\u)$ converges to a flock with fast alignment.
\end{proof}

\section{Strong solutions exist for sub-critical non-vacuum initial data}\label{sec:nonvacuum}

\subsection{General considerations}
In this section, we discuss the existence of global
strong solutions of the alignment system \eqref{eq:main}. The goal is
to control $\|\grad_\x\u(\cdot,t)\|_{L^\infty}$ for all time.

Let $M=\grad_\x\u$ be the gradient velocity matrix.
Apply gradient operator on both sides of
\eqref{eq:main_2DCSb}, to get
\[M_t+\u\cdot\grad_\x M+M^2=\int_{\RR^n}\grad_\x \phi(|\x-\y|)\otimes
(\u(\y,t)-\u(\x,t))\rho(\y,t)d\y -
M\int_{\RR^n}\phi(|\x-\y|)\rho(\y,t)d\y.\]
Let $~'=\partial_t+\u\cdot\grad_\x$ denote differentiation along the
particle path $\dot{\x}=\u(\x,t)$, then the above system reads
\begin{equation}\label{eq:Mfull}
M'+M^2=\int_{\RR^n}\grad_\x
\phi(|\x-\y|)\otimes(\u(\y,t)-\u(\x,t))\rho(\y,t)d\y -
M\int_{\RR^n}\phi(|\x-\y|)\rho(\y,t)d\y,
\end{equation}
 subject to initial data
$M(\x,0)=\grad\u_0(\x)$.

Instead of working with the specific system \eqref{eq:Mfull} directly, we
consider the following \emph{majorant system},
\begin{equation}\label{eq:M}
M'=-M^2-pM+Q,\quad\text{where }
0<\gamma\leq p\leq\Gamma\ \text{ and } \ |Q_{ij}|\leq \c, i,j=1,\cdots,n.
\end{equation}
Here, $p(\cdot,t)$ and the matrix $Q(\cdot,t)$ are uniformly bounded in terms of the constants  $\gamma, \Gamma$ and $\pm\c$.
As an example,   proposition \ref{prop:fit} below shows that system
\eqref{eq:Mfull} admits a majorant of the type  \eqref{eq:M}, with
\begin{equation}\label{eq:consts}
\gamma=\phi(D)m,~~~ \Gamma=m,~~~\c=\I\Lip{\phi}m.
\end{equation}
\begin{prop}\label{prop:fit}
Suppose $(\rho,\u)$ is a solution of the CS system
\eqref{eq:main_2DCS}. Then, for any $\x\in supp(\rho(t))$,
\begin{eqnarray*}
&\displaystyle\left|\int_{\RR^n}\partial_{x_j}
\phi(|\x-\y|)(u_i(\y,t)-u_i(\x,t))\rho(\y,t)d\y\right|\leq \I\Lip{\phi}m,\quad
i,j=1,\cdots,n,\\
&\displaystyle
\phi(D)m\leq\int_{\RR^n}\phi(|\x-\y|)\rho(\y,t)d\y\leq m.
\end{eqnarray*}
\end{prop}
\begin{proof}
For the first inequality,
\begin{align*}
&\left|\int_{\RR^n}\partial_{x_j}\phi(|\x-\y|)
(u_i(\y,t)-u_i(\x,t))\rho(\y,t)d\y\right|\leq
\int_{\RR^n}|u_i(\y,t)-u_i(\x,t)|\left|\partial_{x_j}\phi(|\x-\y|)\right|
\rho(\y,t)d\y\\
&\qquad\leq
\Vu(t)\Lip{\phi}\int_{\RR^n}\rho(y,t)d\y\leq V_0\Lip{\phi}m
=\I\Lip{\phi}m.
\end{align*}
One half of the second inequality is straightforward
$\displaystyle \int_{\RR^n}\phi(|\x-\y|)\rho(\y,t)d\y\leq\|\phi\|_{L^{\infty}}m=m$.
For the other half, recall that the flocking behavior of  $(\rho,\u)$ implies  the uniform bound $\Vx(t)\leq D$, and hence
\[
\int_{\RR^n}\phi(|\x-\y|)\rho(\y,t)d\y=
\int_{\y\in supp(\rho(t))}\!\!\!\!\!\!\!\!\!\!\!\!\!\!\!\!\phi(|\x-\y|)\rho(\y,t)d\y\geq\phi(D)\int_{\RR^n}\rho(\y,t)d\y
=\phi(D)m,
\]
for all $\x\in supp(\rho(t))$.
\end{proof}

\noindent
We proceed to discuss the regularity of  the CS model in view of  its majorant system \eqref{eq:M}.

\subsection{Flocking in one-dimensional  Cucker-Smale hydrodynamics}
We study  the majorant system \eqref{eq:M} for dimension $n=1$,
in which case, $M=u_x$ is a scalar. Denoting $d(t)=u_x(t)$,
we end up with a Riccati-type scalar equation along particle path
\begin{equation}\label{eq:1Dlambda}
d'=- d^2-p d+Q,\quad
\text{where}~p\in[\gamma, \Gamma],~Q\in[-\c,\c],
\end{equation}
for which we have the following conditional stability, consult \cite{LT} for details.
\begin{prop}[Critical threshold for Riccati-type majorant]\label{prop:1Dwoflock}
Consider initial value problem of \eqref{eq:1Dlambda}. We have the
following:
\begin{itemize}
\item
If $\gamma^2-4\c\geq 0$ and
$ d(0)\geq -(\gamma+\sqrt{\gamma^2-4\c})/2$, then $d(t)$ is
bounded for all time $t\geq 0$.
\item
If $ d(0)<-(\Gamma+\sqrt{\Gamma^2+4\c})/2$, then
$ d(t)\to-\infty$ in finite time.
\end{itemize}
\end{prop}

\noindent
Applying proposition \ref{prop:1Dwoflock} for the CS majorant equation \eqref{eq:1Dlambda} with $\gamma, \Gamma$ and $\c$
given in \eqref{eq:consts} we
derive the  following critical thresholds for one-dimensional CS in the non-vacuum region.

\begin{thm}[1D critical thresholds]\label{thm:1DNF}
Consider one-dimensional CS system  \eqref{eq:main1D}.
If the initial configuration satisfies
\[
\I \leq \frac{\phi^2(D)m}{4\Lip{\phi}}\quad\text{and}\quad
d_0 \geq -\frac{1}{2}\left(\phi(D)m+\sqrt{\phi^2(D)m^2-4\I\Lip{\phi}m}\right),\]
then $u_x(x,t)$ remains uniformly bounded for all $(x,t)\in supp (\rho)$.
On the hand, if $\displaystyle 
d_0<-\frac{1}{2}\left(m+\sqrt{m^2+4\I\Lip{\phi}m}\right)$, 
then there is a finite-time blow-up at $t=T_c$, where  
\[
\inf_{x\in supp(\rho(\cdot,t))} u_x(x,t)\to-\infty \ \ \mbox{as} \ \  t\rightarrow T_c-.
\]
\end{thm}

\begin{rem}
The thresholds in theorem \ref{thm:1DNF} correspond to darker areas in
\ref{sec:ff}, taking into account of the additional fast alignment property.
\end{rem}

\subsection{Flocking in two-dimensional Cucker-Smale hydrodynamics}
We extend the result to two dimensions. Instead of being a scaler, $M$
is a 2-by-2 matrix. The dynamics of $M$ in \eqref{eq:M} reads
\begin{equation}\label{eq:2DM}
\begin{bmatrix}M_{11}&M_{12}\\M_{21}&M_{22}\end{bmatrix}'=
-\begin{bmatrix}M_{11}^2+M_{12}M_{21}&M_{12}(M_{11}+M_{12})\\
M_{21}(M_{11}+M_{12})&M_{22}^2+M_{12}M_{21}\end{bmatrix}
-\begin{bmatrix}pM_{11}&pM_{12}\\pM_{21}&pM_{22}\end{bmatrix}
+\begin{bmatrix}Q_{11}&Q_{12}\\Q_{21}&Q_{22}\end{bmatrix},
\end{equation}
where $p\in[\gamma,\Gamma]$ and $|Q_{ij}|\leq \c$ for $i,j=1,2$.
For Cucker-Smale system \eqref{eq:main1D}, the constants are given in
\eqref{eq:consts}.

To bound the entries of $M_{ij}$, it is natural to employ $d:=\div\u$ which will play the role that $u_x$ had in 
the one-dimension setup. Note that
$d=\div\u$ is the trace $d=\text{tr}(M)=\lambda_1+\lambda_2$, where $\lambda_{1,2}$ are two eigenvalues of $M$.
The remaining difficulty is to bound the other entries
of $M$, namely $q:=M_{11}-M_{22}, r:=M_{12}$ and $s:=M_{21}$.
Expressed in terms of $(d,q,r,s)$, system \eqref{eq:2DM} reads
\begin{subequations}\label{eq:dqrs}
\begin{align}
d'+\frac{d^2+\eta^2}{2} & =-pd+(Q_{11}+Q_{22}),\label{eq:2Dd}\\
q'+q(d+p)&=Q_{11}-Q_{22},\label{eq:2Dq}\\
r'+r(d+p)&=Q_{12},\label{eq:2Dr}\\
s'+s(d+p)&=Q_{21},\label{eq:2Ds}
\end{align}
\end{subequations}
Here, the dynamics of $d$ in (\ref{eq:2Dd}) involves the \emph{spectral gap}
 $\eta:=\lambda_1-\lambda_2$ introduced in \cite{LT3} to characterize the critical thresholds of 2D restricted Euler-Poisson equations.

Observe that if $\eta$ is  uniformly bounded in time,
 $\eta(t)\leq\widetilde{c}$ then (\ref{eq:2Dd}) yields
\begin{equation}\label{eq:2Ddi}
d'=-\frac{d^2}{2}-pd+\widetilde{Q},
\quad p\in[\gamma, \Gamma], \quad
\widetilde{Q}:=Q_{11}+Q_{22}-\frac{\eta^2}{2}\in\left[-2\c-\frac{\widetilde{c}}{2}, 2\c+\frac{\widetilde{c}}{2}\right]
\end{equation}
which is a one dimensional equation of the type \eqref{eq:M}.
Therefore, as argued in \cite{LT3}, a bound on $\eta$ is at the heart of matter for two dimensional critical threshold. to this end we note the relation $\eta^2\equiv q^2+4rs$ and hence, if $(q,r,s)$ are bounded, so is
the spectral gap, $\eta$. This is the content of our next lemma.

\begin{lem}[Uniform bound for the spectral gap]\label{lem:gap}
Suppose $(q,r,s)$ are bounded initially by
\begin{equation}\label{eq:qrs}
\max\{|q(0)|, 2|r(0)|,2|s(0)|\}\leq B.
\end{equation}
If $d(t)\geq-\gamma+2\c B^{-1}$ for $t\in[0,T]$, then
the  $(q,r,s)$ remain bounded
\[\max\{|q(t)|, 2|r(t)|,2|s(t)|\}\leq B,~~\text{for}~t\in[0,T].\]
Moreover, the spectral gap $|\eta(t)|\leq\sqrt{2}B$ is also bounded
for $t\in[0,T]$.
\end{lem}
\begin{proof}
We prove the result for $q$ by contradiction.
Suppose there exists a (smallest)
$t_0\in[0,T]$ such that $|q(t)|>B$ for $t\in(t_0, t_0+\delta)$. By
continuity, $|q(t_0)|=B$. There are two cases.
\begin{itemize}
\item $q(t_0)=B, q'(t_0)>0$. Then
$q'(t_0)+q(t_0)(d(t_0)+p)>0+B(2\c B^{-1})=2\c$. This
contradicts with \eqref{eq:2Dq} as $Q_{11}-Q_{22}\leq 2\c$.
\item $q(t_0)=-B, q'(t_0)<0$. Then
$q'(t_0)+q(t_0)(d(t_0)+p)<0-B(2\c B^{-1})=-2\c$. This
also contradicts with \eqref{eq:2Dq} as $Q_{11}-Q_{22}\geq -2\c$.
\end{itemize}
Therefore, $|q(t)|\leq B$ for $t\in[0,T]$.
Same argument yields the boundedness of $r$ and $s$.
Finally, $|\eta(t)|=\sqrt{|q^2(t)+4r(t)s(t)|}\leq\sqrt{2}B$, for $t\in[0,T]$.
\end{proof}

Lemma \ref{lem:gap} tells us  that the spectral gap $\eta$ is bounded as long as the divergence $d$ is not too negative. Under this assumption, the majorant equation \eqref{eq:2Ddi} holds with 
$\widetilde{Q}\in[-2\c-B^2,2\c+B^2]$ and proposition
\ref{prop:1Dwoflock} then yields  the following result.

\begin{prop}\label{prop:2D} Let $B$ denote the bound of \eqref{eq:qrs} and
assume $d(0)\geq-\gamma+\sqrt{\gamma^2-4\c-2B^2}\geq-\gamma+2\c B^{-1}.$
Then $M_{ij}$ are uniformly bounded for all time.
\end{prop}
\begin{proof}
We claim that
$d(t)\geq-\gamma+\sqrt{\gamma^2-4\c-2B^2}\geq-\gamma+2\c
B^{-1}$ and $ \max\{|q(t)|, 2|r(t)|,2|s(t)|\}\leq B$.
Indeed, violation of the first condition contradicts proposition
\ref{prop:1Dwoflock}. Violation of the second condition contradicts
lemma \ref{lem:gap}.
\end{proof}
\begin{rem}
We can rewrite the assumption in proposition \ref{prop:2D} as follows:
\[d(0)\geq-\gamma+\sqrt{\gamma^2-4\c-2B^2},\quad
B\leq\frac{1}{2}\sqrt{(\gamma^2-4\c)+\sqrt{(\gamma^2-4\c)^2-32\c^2}}.\]
Therefore, to ensure boundedness of $M$ in all time, we need $d(0)$ not too
negative, and $q(0), r(0), s(0)$ small.
\end{rem}

We now combine these estimate across the fan of all particle paths. With Cucker-Smale setup
\eqref{eq:consts}, we conclude the following theorem.

\begin{thm}[2D critical thresholds]\label{thm:2DNF}
Consider the two-dimensional CS system  \eqref{eq:main_2DCS}.
 If the initial configuration satisfies the following three estimates
\[
\I\leq\frac{(\sqrt{2}-1)\phi^2(D)m}{4\Lip{\phi}}, \quad
d_0 \geq-\frac{1}{2}\left(\phi(D)m+\sqrt{\phi^2(D)m^2-4\I\Lip{\phi}m-
2B_0^2}\right),
\]
\[
\K\leq\frac{1}{2}\sqrt{\phi^2(D)m^2-4\I\Lip{\phi}m+\sqrt{
\left(\phi^2(D)m^2-4\I\Lip{\phi}m\right)^2-32\I^2\Lip{\phi}^2m^2}},
\]
then $\grad_\x\u(\x,t)$ remains uniformly bounded for all $(\x,t)\in supp (\rho)$.\newline
if, on the other hand, 
\[
d_0<-\frac{1}{2}\left(m+\sqrt{m^2+4\I\Lip{\phi}m}\right),\]
\[|\partial_{x_2}u_{01}|,|\partial_{x_1}u_{02}|\geq\frac{\I\Lip{\phi}m}
{\sqrt{m^2+4\I\Lip{\phi}m}},~~\text{and}~~
\partial_{x_2}u_{01}\cdot\partial_{x_1}u_{02}>0,\]
then there is a finite time blowup $t=T_c>0$ such that
 $\inf_{\x\in supp(\rho(\cdot,t))}\div\u(x,t)\to-\infty$ as $t\rightarrow T_c-$.
\end{thm}

\begin{rem}
Note that if $\K=0$, the above result is reduced to the one dimensional case. In general, the bound on $B$ (and hence on the spectral gap), restricts the range of sub-critical  $d(0)$, while still keeping the relevant range to include negative initial divergence.
\end{rem}
\begin{rem}
 We provide a critical threshold of
the initial profile which leads to a  finite time break down. The idea
and the result are similar to the one-dimensional case.
The additional assumptions on $\partial_{x_2}u_{01}$ and $\partial_{x_1}u_{02}$ made in the second part of the theorem,  guarantee that  the spectral gap $\eta(\cdot,t)$ is real for all time; we omit the proof, as  it does not prevent $d$ from a finite-time blowup.
\end{rem}

\section{Fast alignment in Cucker-Smale hydrodynamics}\label{sec:fast}
\subsection{General considerations}\label{sec:ff}
In this subsection, we introduce a new prototype of problems to
characterize the dynamics of $M$, taking advantage of the fast
alignment property.

From the proof of proposition \ref{prop:fit}, we have
\[\left|\int_{\RR^n}\partial_{x_j}
\phi(|\x-\y|)(u_i(\y,t)-u_i(\x,t))\rho(\y,t)d\y\right|\leq \Vu(t)\Lip{\phi}m,\quad
i,j=1,\cdots,n.\]
Instead of taking the rough maximum principle bound $\Vu(t)\leq
\Vu(0)=\I$, we make use of the much stronger fast alignment property
$\displaystyle \frac{d}{dt}\Vu(t)\leq -m\phi(D)\Vu(t)$.
It leads to the following prototype of majorant system
\begin{subequations}\label{eq:Mdv}
\begin{align}
M'&= -M^2-pM+\Vu Q, \qquad
0<\gamma\leq p\leq\Gamma, \quad |Q_{ij}|\leq\c,
i,j=1,\cdots,n.\\
\displaystyle \frac{d}{dt}\Vu & \leq  -G\Vu.  \label{eq:ffdv}
\end{align}
\end{subequations}
Such majorant system holds for Cucker-Smale equations \eqref{eq:main_2DCS},with
\begin{equation}\label{eq:constff}
\gamma=\phi(D)m, ~~\Gamma=m, ~~C=\Lip{\phi}m, ~~G=\phi(D)m.
\end{equation}

\noindent
We now couple the dynamics of $M$ to  the dynamics of $V$. Because of   fast alignment, the  (``bad'') term 
 $VQ$ has an exponentially decaying  change on $M$, which enables larger  set of sub-critical initial
configurations which ensure the boundedness of $M$ (illustrated in figure \ref{fig:oned}).

\subsection{One-dimensional flocking with fast alignment}\label{sec:1DFF}
We study the evolution of system \eqref{eq:Mdv} in one-dimension,
where $ d=M$ is a scalar. The $2\times 2$ system reads
\begin{subequations}\label{eq:lambdadv}
\begin{align}
d'&=- d^2-p d+c\Vu, \qquad
p\in[\gamma,\Gamma],~~~ c\in[-C,C].\\
\displaystyle \frac{d}{dt}\Vu &\leq  -G\Vu.
\end{align}
\end{subequations}

The following theorem characterizes the dynamics of $( d, \Vu)$.
\begin{thm}\label{thm:lambdadv}
Consider the $2\times 2$ majorant system \eqref{eq:lambdadv}. 
 There exists an upper threshold function $\ThU:\RR^+\to[-\gamma,+\infty)$,
  defined implicitly as
\begin{equation}\label{eq:fr}
\ThU(0)=-\gamma,~ \quad \ThU'(x)=
\begin{cases}
\displaystyle \frac{C}{\gamma+G},&x\to0+\\
\displaystyle \frac{-\ThU^2(x)-\gamma \ThU(x)-Cx}{-Gx}& \text{if}~\ThU(x)<0\\
\displaystyle \frac{-\ThU^2(x)-\Gamma \ThU(x)-Cx}{-Gx}& \text{if}~\ThU(x)\geq0
\end{cases}
\end{equation}
such that, if $ d(0)>\ThU(V_0)$ for all $x$, i.e. if $(V_0, d(0))$
lies above $\ThU$, then $(\Vu(t), d(t))$ remain bounded for all
time, and $ d(t)\to0, \Vu(t)\to0$ as $t\to\infty$.\newline
On the other hand, there exists a lower threshold  function $\ThL:~\RR^+\to(-\infty,-\Gamma]$,
\begin{equation}\label{eq:fl}
\ThL(0)=-\Gamma,~ \quad \ThL'(x)=
\begin{cases}
-\displaystyle \frac{C}{\Gamma+G},& x\to0+\\
\displaystyle \frac{-\ThL^2(x)-\Gamma \ThL(x)+Cx}{-Gx}
&x>0.
\end{cases}
\end{equation}
such that, if  $d(0)<\ThL(V_0)$, i.e. $(V_0, d(0))$
lies below $\ThL$, then $ d(t)\to-\infty$ at a finite time.
\end{thm}

Apply theorem \ref{thm:lambdadv} to system \eqref{eq:main_2DCS} by plugging
in the values of the constants given in \eqref{eq:constff} and combine across the fan of all particle paths. Theorem \ref{thm:1DCS} follows with
\begin{equation}\tag{Cucker-Smale: $\ThU$}
\ThU(0)=-\phi(D)m,~ \ThU'(x)=
\begin{cases}
\displaystyle \frac{\Lip{\phi}}{2\phi(D)}&x\to0+\\
\displaystyle \frac{-\ThU^2(x)-\phi(D)m \ThU(x)-\Lip{\phi}mx}{-\phi(D)mx}&\text{if}~\ThU(x)<0\\
\displaystyle \frac{-\ThU^2(x)-m\ThU(x)-\Lip{\phi}mx}{-\phi(D)mx}&\text{if}~\ThU(x)\geq0
\end{cases},
\end{equation}
\begin{equation}\tag{Cucker-Smale: $\ThL$}
\ThL(0)=-m,~ \ThL'(x)=
\begin{cases}
\displaystyle -\frac{\Lip{\phi}}{1+\phi(D)}& x\to0+\\
\displaystyle \frac{-\ThL^2(x)-m\ThL(x)+\Lip{\phi}mx}{-\phi(D)mx}
&x>0
\end{cases}.
\end{equation}

\begin{rem}
Comparing theorem \ref{thm:1DNF} and theorem
\ref{thm:1DCS} we see that the additional fast alignment property enables us to establish a much larger area of sub-critical $(\I,\J)$ for which  $u_x$ remains bounded in the non-vacuum area. In particular,
an upper bound of $\I$ is not required any more .
\end{rem}

\subsection{Proof of theorem \ref{thm:lambdadv}}
The proof of the theorem can be separate into two parts. First, we
discuss the evolution of the majorant system
\begin{subequations}\label{eq:omegaeta}
\begin{align}
\frac{d}{dt}\omega&=-\omega^2-E\omega+F\eta,\\
\frac{d}{dt}\eta&=-G\eta,
\end{align}
where $E>0,F\in\RR,G>0$ are constant coefficients.
\end{subequations}
Then, we state a comparison principle, comparing $( d,\Vu)$ with
$(\omega,\eta)$ and  derive critical thresholds for the evolution of the inequality system \eqref{eq:lambdadv}.

\begin{prop}[Critical threshold for the majorant system (\ref{eq:omegaeta})]\label{prop:equal}
Suppose $(\eta(t),\omega(t))$ satisfy \eqref{eq:omegaeta}  with initial condition $\omega(0)=\omega_0$,
$\eta(0)=\eta_0>0$. Then, there exists a separatrix curve $f(\cdot)$ such that
\begin{itemize}
\item If $\omega_0> f(\eta_0)$, i.e. $(\eta_0,\omega_0)$ lies above
  $f$, we have $\omega(t)\to 0, \eta(t)\to
  0$ as $t\to\infty$,
\item If $\omega_0=f(\eta_0)$, i.e. $(\eta_0,\omega_0)$ lies on
  $f$, we have $\omega(t)\to -E, \eta(t)\to
  0$ as $t\to\infty$,
\item If $\omega_0< f(\eta_0)$, i.e. $(\eta_0,\omega_0)$ lies below
  $f$, we have $\omega(t)\to -\infty$ as $t\to\infty$.
\end{itemize}
The separatrix  $f$ is implicitly defined below,
\begin{equation}\label{eq:1Df}
f(0)=-E,~ f'(0)=-\frac{F}{E+G},~ f'(x)=\frac{-f^2(x)-Ef(x)+Fx}{-Gx}~
\text{ for }~x\in(0,+\infty).
\end{equation}
\end{prop}
\begin{proof}
The linearized system associated with (\ref{eq:omegaeta}) has two stationary points --- a stable point at $O(0,0)$ and a saddle at $A(0,-E)$. 
Figure \ref{fig:phase} shows the phase plane of $(\eta,\omega)$.
A critical curve $f$, starting from $A$ and traveling 
along the vector field,  divides the plane $\RR^+\times\RR$ into
two parts. Flows starting above $f$ converge to the stable point
$O$, while flows starting below $f$ will diverge.

\begin{figure}[ht]
\vspace*{-3.0cm}
\centering
\hspace*{0.5cm} \includegraphics[width=.49\linewidth]{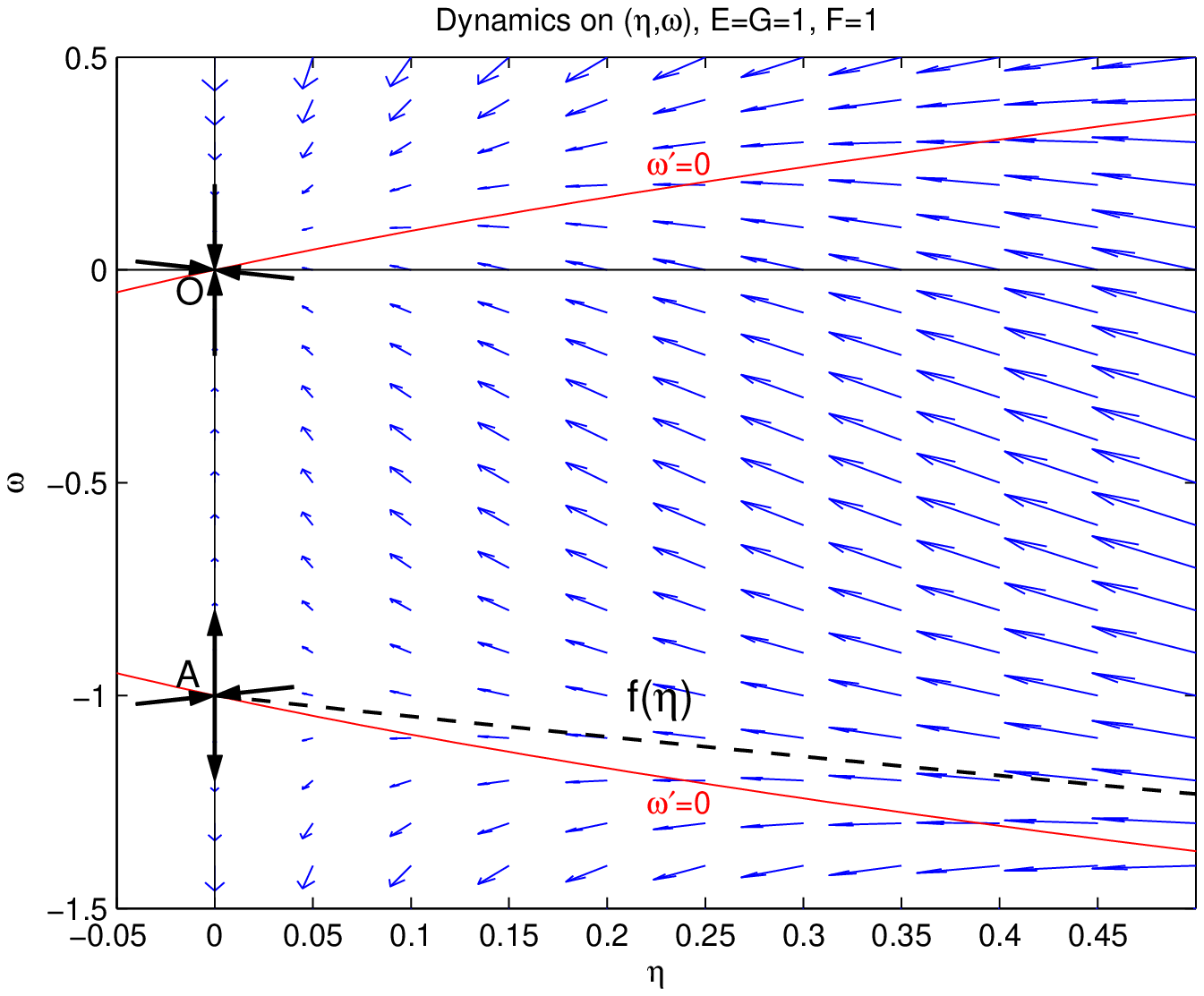}
\includegraphics[width=.49\linewidth]{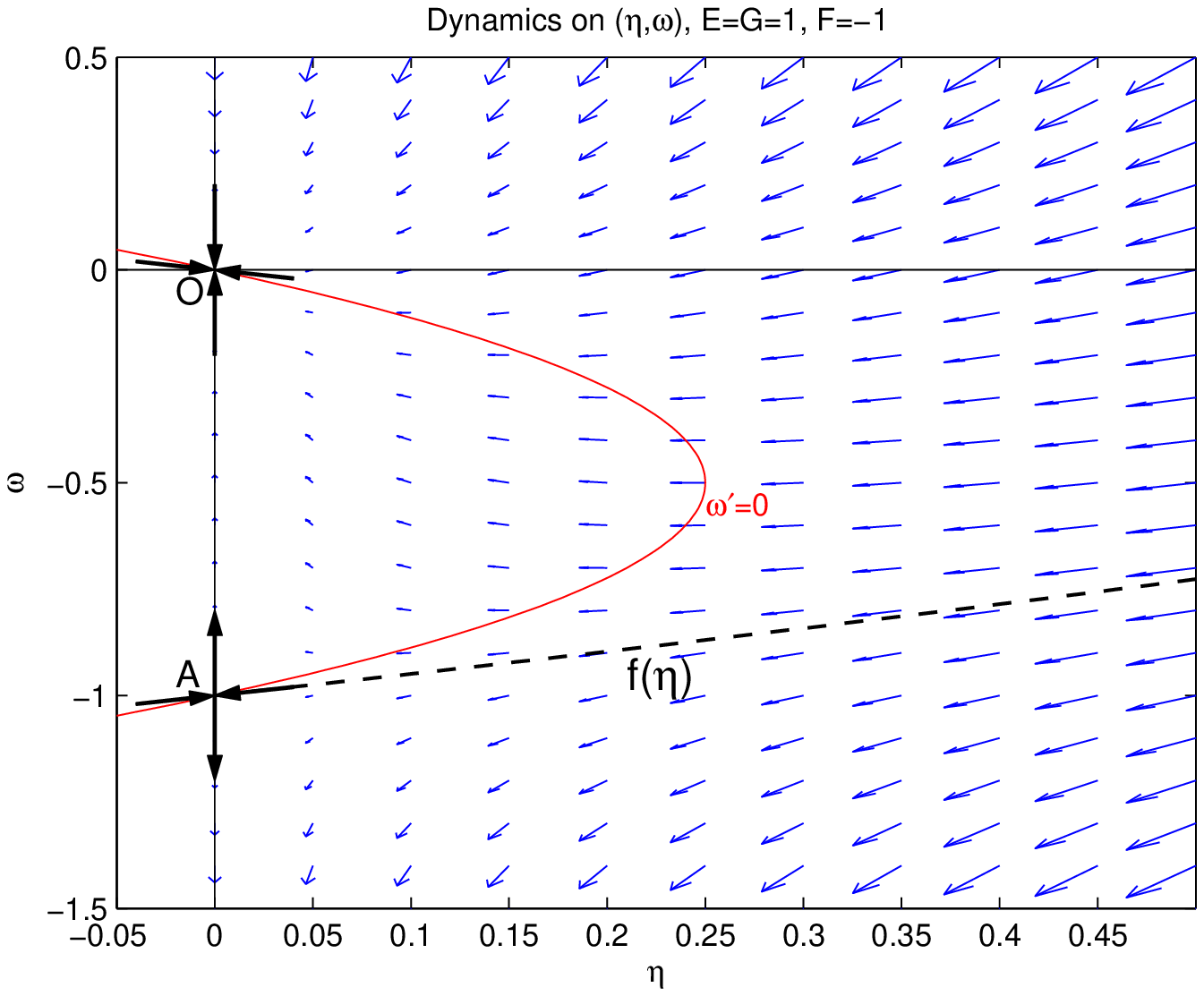}\hspace*{-0.5cm}
\caption{Phase plane of $(\eta,\omega)$ and critical thresholds for
  $F>0$(left) and $F<0$(right)}\label{fig:phase}
\vspace*{3.5cm}
\end{figure}

Along the line, clearly we have
\[f'(x)=\frac{d\omega}{d\eta}=\frac{\frac{d}{dt}\omega}{\frac{d}{dt}\eta}
=\frac{-\omega^2-E\omega+F\eta}{-G\eta}=\frac{-f^2(x)-Ef(x)+Fx}{-Gx}.\]
When $x\to 0$, we have
$ f'(0)=\lim_{x\to0}\frac{-f^2(x)-Ef(x)+Fx}{-Gx}
=\frac{-Ef'(0)-F}{G}$ and (\ref{eq:1Df}) follows.
\end{proof}

The following lemma states the relationship between the solution of
the equality system \eqref{eq:omegaeta} and the inequality system
\eqref{eq:lambdadv}. It allows us
to extend the critical thresholds result to the
inequality system.

\begin{lem}[A comparison principle]\label{thm:compare}
Let $( d,\Vu)$ satisfy \eqref{eq:lambdadv}, and $(\omega,\eta)$
satisfy \eqref{eq:omegaeta} with $E, F$ defined as below.
$t_0\geq 0$ and $T\in(t_0,+\infty]$.
\begin{enumerate}
\item Suppose $\omega(t)\geq0$, for $t\in[t_0,T]$.
\begin{itemize}
\item[(1a)] Let $E=\gamma, F=C$. \ 
If $\displaystyle\begin{cases} d(t_0)\leq\omega(t_0)\\
  \Vu(t_0)\leq\eta(t_0)\end{cases}$, then
$\displaystyle\begin{cases} d(t)\leq\omega(t)\\
  \Vu(t)\leq\eta(t)\end{cases}$ for $t\in[t_0,T]$.
\item[(1b)] Let $E=\Gamma, F=-C$.
If $\displaystyle\begin{cases} d(t_0)\geq\omega(t_0)\\
  \Vu(t_0)\leq\eta(t_0)\end{cases}$, then
$\displaystyle\begin{cases} d(t)\geq\omega(t)\\
  \Vu(t)\leq\eta(t)\end{cases}$ for $t\in[t_0,T]$.
\end{itemize}
\item Suppose $\omega(t)\leq 0$, for $t\in[t_0,T]$.
\begin{itemize}
\item[(2a)] Let $E=\Gamma, F=C$.
If $\displaystyle\begin{cases} d(t_0)\leq\omega(t_0)\\
  \Vu(t_0)\leq\eta(t_0)\end{cases}$, then
$\displaystyle\begin{cases} d(t)\leq\omega(t)\\
  \Vu(t)\leq\eta(t)\end{cases}$ for $t\in[t_0,T]$.
\item[(2b)] Let $E=\gamma, F=-C$.
If $\displaystyle\begin{cases} d(t_0)\geq\omega(t_0)\\
  \Vu(t_0)\leq\eta(t_0)\end{cases}$, then
$\displaystyle\begin{cases} d(t)\geq\omega(t)\\
  \Vu(t)\leq\eta(t)\end{cases}$ for $t\in[t_0,T]$.
\end{itemize}
\end{enumerate}
\end{lem}
\begin{proof}
We only prove (1a); the  other cases are similar.
Subtracting \eqref{eq:lambdadv} with \eqref{eq:omegaeta}, we get
\begin{eqnarray*}
\frac{d}{dt}(\omega- d)\geq-(\omega+ d)(\omega- d)
-p(\omega- d)+C(\eta-\Vu), \qquad 
\frac{d}{dt}(\eta-\Vu)\geq-G(\eta-\Vu).
\end{eqnarray*}

Suppose by contradiction $\Vu(t)>\eta(t)$ for some $t\in(t_0,T)$. As $\Vu,\eta$ are
continuous, there exists $\tau\in(t_0,t)$ such that $\eta(\tau)-\Vu(\tau)=0$
and $\frac{d}{dt}(\eta(\tau)-\Vu(\tau))<0$. This violates the second
inequality. So, $\Vu(t)\leq\eta(t)$ for all $t\in[t_0,T]$.
Similarly, suppose by contradiction $ d(t)>\omega(t)$ for some
$t\in(t_0,T)$. As $\Vu,\eta$ are
continuous, there exists $\tau\in(t_0,t)$ such that $ d(\tau)-\omega(\tau)=0$
and $\frac{d}{dt}( d(\tau)-\omega(\tau))<0$. Meanwhile,
$\Vu(\tau)\leq\eta(\tau)$.
This violates the first inequality. So, $ d(t)\leq\omega(t)$ for all $t\in[t_0,T]$.
\end{proof}

We can use this comparison principle to verify theorem \ref{thm:lambdadv}.
First, $d$ remains  bounded from above. Indeed, suppose by contradiction that $ d(t)\to+\infty$ as $t\to T$. Then, there exists a $t_0\in[0,T)$ such that $ d(t_0)>0$. Construct
$(\omega,\eta)$ by \eqref{eq:omegaeta} with $E=\gamma$, $F=C$ and with
initial values $\omega(t_0)= d(t_0)>0, \eta(t_0)=\Vu(t_0)$. But according to
proposition \ref{prop:equal}, $\omega(t)$ is bounded from above and
comparison principle (1a) implies that $ d(t)\leq\omega(t)$
is also upper bounded.
To prove that $d$ is bounded from below, we apply the same comparison
argument. If $\ThU(x)<0$, we use (2b). If $\ThU(x)\geq0$, we use
(1b). Details are left to reader. For lower threshold $\ThL$, we prove that $d(t)\to-\infty$ in a finite  time using comparison principle (2a). Again, we omit the details.

\subsection{Two-dimensional flock with fast alignment}
In this subsection, we invoke the fast alignment property
to derive critical thresholds determined by initial
quantities $(\I,\J,\K)$, which are more relaxed  than those in 
theorem \ref{thm:2DNF}.

First, rewrite system \eqref{eq:dqrs} coupled with fast decay property \eqref{eq:ffdv}.
\begin{subequations}\label{eq:dqrs2}
\begin{align}
d'+\frac{d^2+\eta^2}{2}& = -pd+(Q_{11}+Q_{22}),\label{eq:ffd}\\
q'+q(d+p)&=(Q_{11}-Q_{22})\Vu,\label{eq:ffq}\\
r'+r(d+p)&=Q_{12}\Vu,\label{eq:ffr}\\
s'+s(d+p)&=Q_{21}\Vu,\label{eq:ffs}\\
\frac{d}{dt}\Vu&=-G\Vu,\label{eq:ffdv2}
\end{align}
\end{subequations}
where $p\in[\gamma,\Gamma]$ and $|Q_{ij}|\leq c$ for $i,j=1,2$.
We now state the uniform boundedness result for the spectral gap $\eta$.
\begin{lem}\label{lem:gap2}
Let $b_0=\max\{|q(0)|,2|r(0)|,2|s(0)|\}$. Suppose there exists a
positive constant $\delta$ such that $d(t)\geq-\gamma+\delta$ for all
$t\geq 0$. Then there exists a threshold $\zeta=\zeta(V_0; \delta,B)$ such that if  $b_0\leq \zeta$, then $(q,r,s)$ are uniformly bounded,
$\max\{|q(0)|,2|r(0)|,2|s(0)|\}\leq B$.
\end{lem}

The details of the function $\zeta$ are given below
\begin{equation}\label{eq:g}
\zeta(x;\delta,B)=\begin{cases}
\displaystyle B & \displaystyle x\in\left[0,\frac{\delta B}{2C}\right]\\
\displaystyle \frac{B}{\delta-G}\left[-G\left(\frac{2C}{\delta B}x\right)^{\delta/G}
+\frac{2C}{\delta B}x\right]& \displaystyle x\in\left[\frac{\delta B}{2C},
\left(\frac{\delta}{G}\right)^{\frac{G}{\delta-G}}\frac{\delta
  B}{2C}\right],~~\delta\neq G\\
\displaystyle \frac{2C}{\delta}\left(1-\log\left(\frac{2C}{\delta
      B}x\right)\right)x&
\displaystyle x\in\left[\frac{\delta B}{2C}, \frac{\delta B e}{2C}\right],~~\delta=G
\end{cases}.
\end{equation}

Lemma \ref{lem:gap2} provides a region in phase space of  $(b_0, V_0)$
such that the spectral gap is uniformly bounded in all time.
From the definition of $\zeta$, we observe that, to guarantee a uniform
upper bound,  $B$, $V_0$ can not be too large. Given $\delta$ and
$B$, the upper bound of $V_0$ is
$\left(\frac{\delta}{G}\right)^{\frac{G}{\delta-G}}\frac{\delta
  B}{2C}$
,independent of the choice of $b_0$.
\begin{proof}
We prove the result for $q$. Consider  the
coupled system \eqref{eq:ffq} and \eqref{eq:ffdv2}. The corresponding
majorant system reads
$\omega'=-\delta\omega+2C\eta, \ \eta'=-G\eta$.
This system can be easily solved. Figure \ref{fig:twoDIJK}(a) shows the
dynamics of $(\eta,\omega)$. The filled area includes all initial
conditions such that $\omega(t)\leq B$ for all $t\geq 0$. The area is
governed by a function $g$. A simple computation yields an explicit
expression of $g$, which is stated in \eqref{eq:g}.

A comparison argument enable us to connect the equality system with
the inequality system, which says
\[\text{If}~~\begin{cases}
|q(0)|\leq\omega(0)\\ V_0\leq\eta(0)
\end{cases}~~~~\text{then}~~~~\begin{cases}
|q(t)|\leq\omega(t)\\ \Vu(t)\leq\eta(t)\end{cases},~~~\text{for all}~t\geq0.\]
Therefore, $|q|$ is bounded by $B$ uniformly in time as long as
$(V_0,|q(0)|)$ lies inside the area, i.e. $|q(0)|\leq
g(V_0)$. Similarly, we prove for $r$ and $s$ which ends the proof.
\end{proof}

\begin{figure}[ht]
\vspace*{-3.8cm}
\begin{centering}
\hspace*{0.6cm} \includegraphics[width=.49\linewidth]{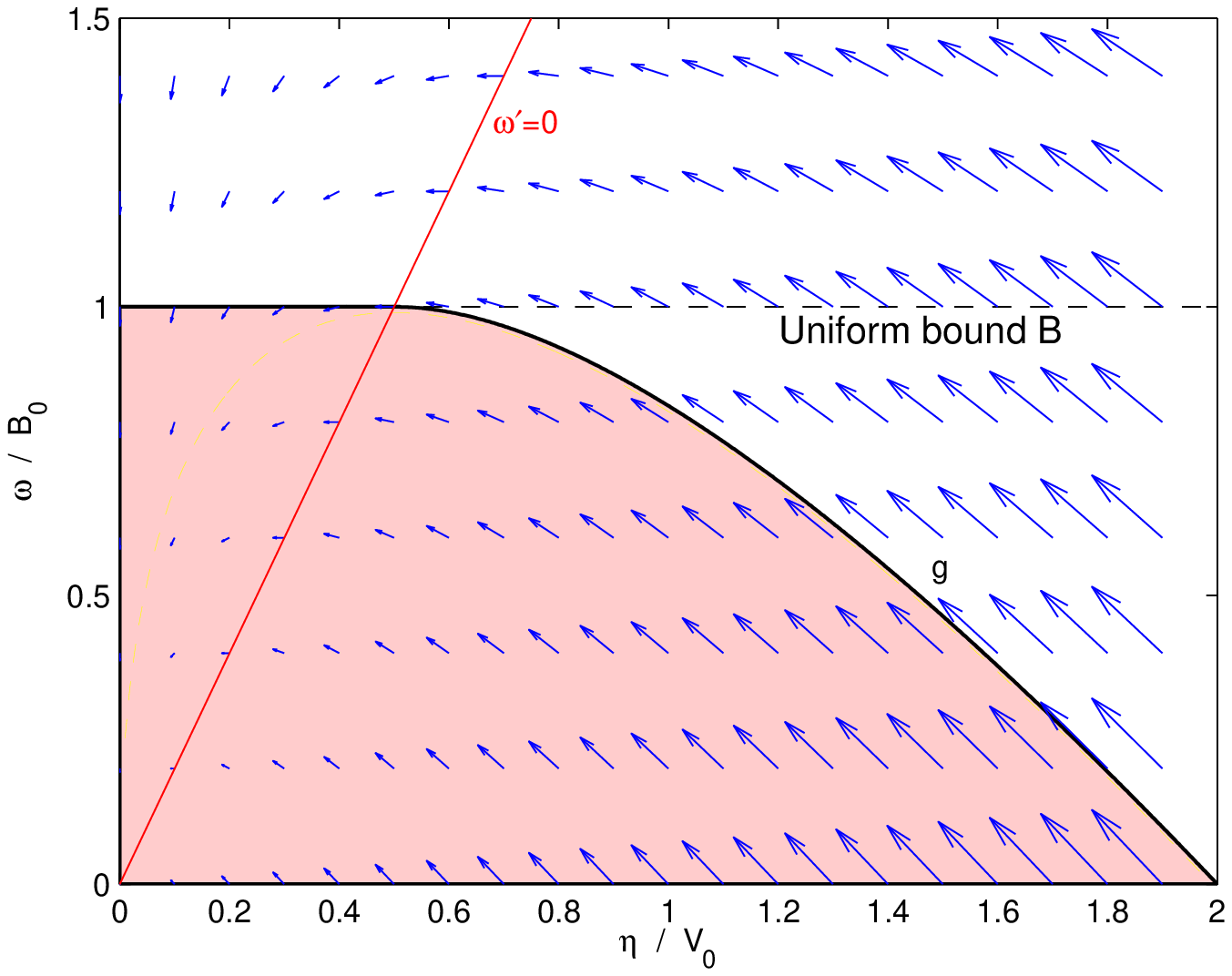}
\includegraphics[width=.49\linewidth]{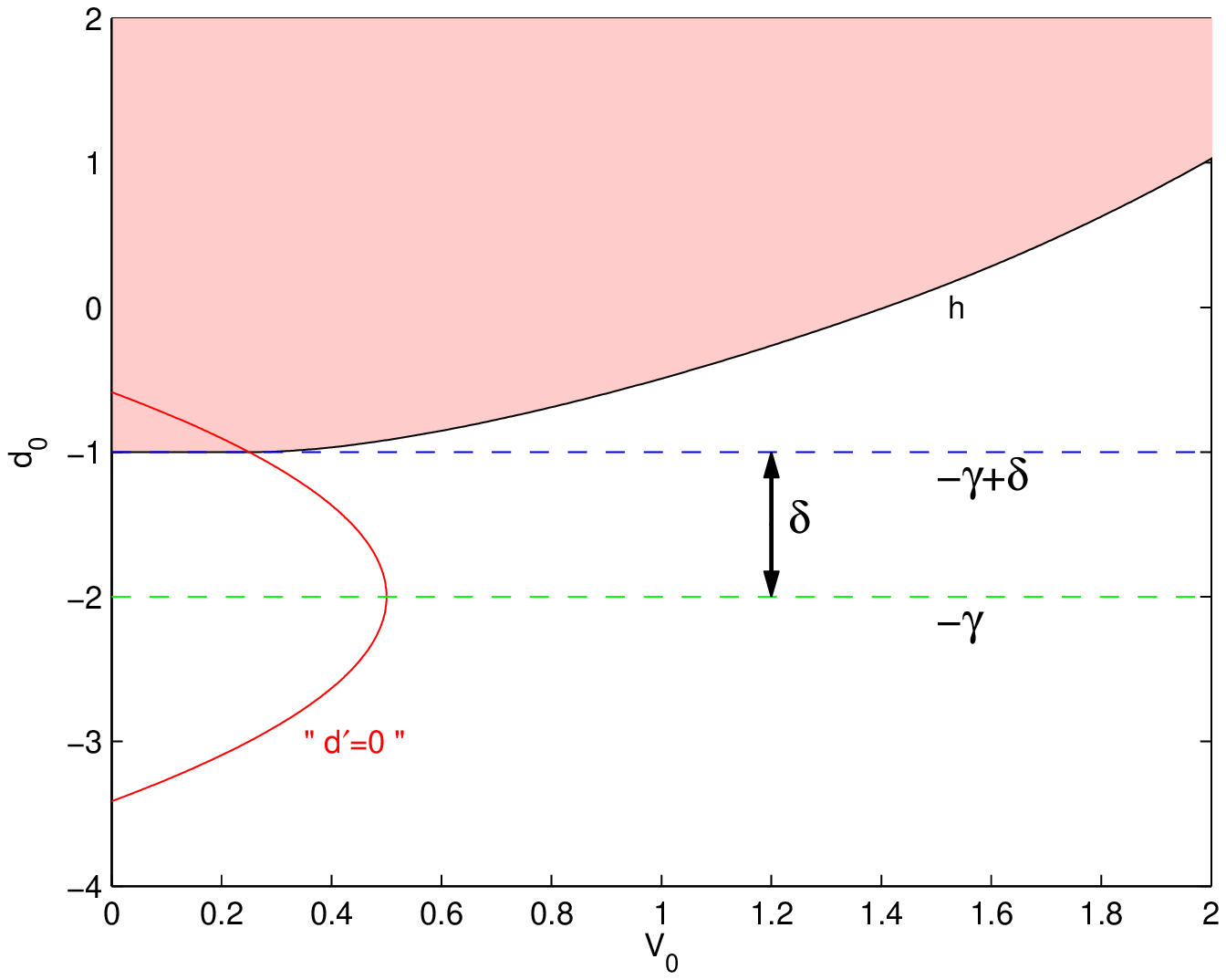}\hspace*{-0.8cm}
\end{centering}
\vspace*{4.2cm}
\caption{(a). Dynamics of the system $(\eta,\omega)$\quad
(b). Dynamics of the system $(\Vu,d)$}\label{fig:twoDIJK}
\end{figure}

Next, for given $\delta$ and $B$, we consider the coupled system \eqref{eq:ffd} and
\eqref{eq:ffdv2} and find the region of $(V_0,d(0))$ such that
$d(t)\geq -\gamma+\delta$.

\begin{prop}\label{prop:IJ}
Suppose there exists a $B$ such that $|\eta(t)|\leq B \leq \gamma/\sqrt{2}$ for $t\geq0$ and let  $\delta\in(0,\sqrt{\gamma^2-2B^2}]$.
Then  there exists a threshold $\ThU= \ThU (V_0;\delta,B)$ such that 
if $d(0)\geq \ThU$, then $d(t)$ remains uniformly bounded in time, and
$d(t)\geq-\gamma+\delta$ for  $t\geq0$.
The upper threshold  $\ThU$ is  defined implicitly
\begin{subequations}\label{eq:h}
\begin{align}
\ThU(x;\delta,B)=&-\gamma+\delta,\qquad x\in\left[0,
\displaystyle \frac{\gamma^2-\delta^2-2B^2}{4C}\right)\\
\displaystyle\ThU'(x;\delta,B)=&\begin{cases}
\displaystyle\frac{\ThU^2(x)+2\gamma\ThU (x)+4Cx+2B^2}{2Gx}&
\text{if}~~\ThU (x)<0,\\
\displaystyle\frac{\ThU^2(x)+2\Gamma \ThU (x)+4Cx+2B^2}{2Gx}&
\text{if}~~\ThU (x)\geq0,
\end{cases}
\displaystyle\quad x\in\left [\frac{\gamma^2-\delta^2-2B^2}{4C},+\infty\right).
\end{align}
\end{subequations}
\end{prop}

Similar to the one-dimensional case, proposition \ref{prop:IJ} can be
easily proved by analyze on the equality system and a comparison rule.
Figure \ref{fig:twoDIJK}(b) shows the area of $(V_0,d(0))$ such that
$d(t)$ is lower bounded by $-\gamma+\delta$ for all time. The area is
governed by $h$ defined in \eqref{eq:h}. We omit the details of the proof.

\begin{thm}[2D Critical Thresholds]\label{thm:2D}
Consider the two-dimensional CS system \eqref{eq:main_2DCS}  with majorant systems involving the constants $(\gamma,\Gamma,C,G)$ given  in \eqref{eq:constff}.
 If there exists $(\delta, B)$ such that
  $\delta^2+2B^2\leq\gamma^2$,
and the initial profiles $(\I,\J,\K)$ satisfies

\noindent
(i) $\K\leq \zeta(\I;\delta,B)$, where $\zeta$ is defined in \eqref{eq:g}, and
(ii) $\J\geq \ThU (\I;\delta,B)$, where $\ThU$ is defined in \eqref{eq:h},\newline
then $|\grad\u(x,t)|$ remains bounded all  $(x,t)\in supp(\rho)$.
\end{thm}

\begin{rem}
The theorem  guarantees the boundedness of $\grad\u$ provided 
 $\K$ is not too large and $\J$ is not  too negative.
\end{rem}

\section{Strong solutions in the presence of vacuum}\label{sec:vacuum}
In this section, we discuss the boundedness of $\grad_\x\u$ when
$(\x,t)\not\in supp (\rho)$. The result allows us to study the system
in whole space, without worrying about the free boundary. It
also extends the global existence result to initial density which is supported over disconnected blobs.
 In the case of standard local models, lack of any relaxation inside the vacuum enables the solution to form  shock discontinuities in
finite time. In the present setup, however, nonlocal alignment prevents the formation of  shock discontinuities.

\subsection{Dynamics inside the vacuum}
Consider the dynamics of $\grad_\x\u$ \eqref{eq:M} for $\x\not\in
supp(\rho_0)$. Define maximum
diameter  of the velocity field between a point in the whole space and
a point in the non-vacuum area
\[V^\infty(t):=\sup\{|\u(\x,t)-\u(\y,t)|,~\x\in\RR^n,~\y\in supp(\rho(\cdot,t))\}.\]
and the distance between $\x$ and the non-vacuum region
$L(\x,t):=dist(\x, supp (\rho(\cdot,t))$. 
We have the
following bounds (to be compared with those in proposition \ref{prop:fit}).
\begin{prop}[Bounds inside the vacuum]\label{prop:fitvacuum}
Suppose $(\rho,\u)$ is a solution of system
\eqref{eq:main}. Then, for any $\x\not\in supp(\rho(t))$,
\begin{eqnarray*}
&\displaystyle\left|\int_{\RR^n}\partial_{x_j}
\phi(|\x-\y|)(u_i(\y,t)-u_i(\x,t))\rho(\y,t)d\y\right|\leq
V^\infty(0)\left|\phi'(L(\x,t))\right|m,\quad
i,j=1,\cdots,n,\\
&\displaystyle
\phi(L(\x,t)+D)m\leq\int_{\RR^n}\phi(|\x-\y|)\rho(\y,t)d\y\leq m.
\end{eqnarray*}
\end{prop}
\begin{proof}
For the first inequality,
\begin{align*}
&\left|\int_{\RR^n}\partial_{x_j}\phi(|\x-\y|)
(u_i(\y,t)-u_i(\x,t))\rho(\y,t)d\y\right|\leq
\int_{supp(\rho(t))}|u_i(\y,t)-u_i(\x,t)|\left|\partial_{x_j}\phi(|\x-\y|)\right|
\rho(\y,t)d\y\\
&\qquad\leq
V^\infty(t)|\phi'(L(\x,t))|\int_{supp(\rho(t))}\rho(y,t)d\y\leq
V^\infty(0)|\phi'(L(\x,t))|m.
\end{align*}
The last inequality is valid due to maximum principle.
For the second inequality,
$\displaystyle \int_{\RR^n}\phi(|\x-\y|)\rho(\y,t)d\y\leq\|\phi\|_{L^{\infty}}m=m$.
On the other hand, as $(\rho,\u)$ converges to a flock, $\Vx(t)$ is
uniformly bounded by $D$, defined in \eqref{eq:D}.
Hence, $\max_{y\in supp(\rho(t))} dist(\x,\y)\leq L(x,t)+D$. It yields
\[\int_{\RR^n}\phi(|\x-\y|)\rho(\y,t)d\y=
\int_{supp(\rho(t))}\phi(|\x-\y|)\rho(\y,t)d\y\geq\phi(L(x,t)+D)m.\]
\end{proof}

\begin{rem}
The key estimate in proposition \ref{prop:fitvacuum}, which distinguishes itself from local system, the is the positive
lower bound of $\phi\star\rho$. 
\end{rem}
Next, we turn to discuss the criterion to guarantee the boundedness of
$\|\grad_\x\u(\cdot,t)\|_{L^\infty}$ in whole space, using similar
technique as Section \ref{sec:nonvacuum}. For simplicity, we
focus on one-dimensional case. Similar result can be easily
established for two dimensions.

\begin{lem}\label{lem:vacuum}
Assume the initial configuration is such that $V^\infty(0)$ and $u_{0x}$ satisfy
\begin{subequations}
\begin{eqnarray}\label{eq:phi2p}
V^\infty(0) & \leq &\inf_{r\geq 0}\left[\frac{m\phi^2(r+D)}{4|\phi'(r)|+2|\phi'(r+D)|}\right],\\
u_{0x}(x) &\geq & -\frac{m}{2}\phi(L(x,0)+D),~~~\text{for
}x\not\in supp(\rho_0).\label{eq:crit}
\end{eqnarray}
\end{subequations}
Then, $u_x(x,t)$ remains bounded for all time for $(x,t)\not\in supp(\rho)$.
\end{lem}

\begin{rem}
Condition \eqref{eq:phi2p} carries two aspects.
(i)  Slow decay at infinity. Suppose $\phi(r)\approx
r^{-\alpha}$ as $r\to\infty$. The right hand side is proportional to $r^{1-\alpha}$.
If $\phi$ decays fast with $\alpha>1$, \textit{i.e.} \eqref{eq:phi2}
is violated, the right hand side goes to 0. The condition can not be
achieved unless $u_0$ is a constant. A slow decay assumption on
$\phi$ is needed to make sure the condition is meaningful.
(ii) Behavior of  $V^\infty$ near the origin.
Take $r=0$, the condition reads
$ \displaystyle V^\infty(0)\leq\frac{m\phi^2(D)}{4\Lip{\phi}+2|\phi'(D)|}$.
This is equivalent to the thresholds of $\I$ in proposition
\ref{prop:1Dwoflock}, assuming $V^\infty(0)\lesssim V_0$.\newline
As for condition \eqref{eq:crit} --- it is satisfied automatically for large $|x|$.
As the matter of fact, when $|x|\to\infty$, \eqref{eq:crit} says that
$u_{0x}(x)\gtrsim-|x|^{-\alpha}$.
This is a consequence of $u_0\in L^\infty(\RR)$ and the fact
$\alpha\leq 1$.
\end{rem}

\begin{proof}[Proof of Lemma \ref{lem:vacuum}.]
Consider $(x,t)\not\in supp(\rho)$. It belongs to a characteristic
starting from $(x_0,0)$ where $x_0\not\in supp(\rho_0)$,
as long as $u_x$ is bounded. At this point, we have
$d'=-d^2-pd+Q$, with $p\in[\phi(L(x,t)+D)m,m]$ and $|Q|\leq V^\infty(0)|\phi'(L(x,t)|m$.

It is then sufficient to discuss the following majorant equation and use the
comparison principle to draw desired conclusion on $d$,
\[\omega'=-\omega^2-\phi(L(x,t)+D)m\omega-V^\infty(0)|\phi'(L(x,t)|m.\]

Condition \eqref{eq:phi2p} ensures that the right hand side has two
distinguished solutions.
In particular, if we pick $\omega=-\frac{1}{2}\phi(L(x,t)+D)m$, then
\[\omega'=\frac{1}{4}\phi^2(L(x,t)+D)m^2-V^\infty(0)|\phi'(L(x,t))|m
\stackrel{\eqref{eq:phi2p}}{\geq}\frac{1}{2}|\phi'(L(x,t)+D)|V^\infty(0) m>0.\]

Let $A(x_0)$ denote the area where
$\omega\geq-\frac{1}{2}\phi(L(x,t)+D)m$, and $(x,t)=(X(t),t)$ is a point on the
characteristic starting from $(x_0,0)$, namely
\[A(x_0):=\left\{(z,t)\left|z\geq-\frac{1}{2}(\phi(L(X(t),t)+D)m),
    t\geq 0\right.\right\}.\]
Its boundary $\partial A(x_0)$ reads
$\partial A(x_0)=\{(\gamma(t),t)|t\geq 0\}$ where
$\gamma(t)=-\frac{1}{2}\phi(L(X(t),t)+D)m$.

Criterion \eqref{eq:crit} implies $(\omega(0),0)\in A(x_0)$. We are left to show
that $(\omega(t),t)$ stays in $A(x_0)$ for all $t\geq 0$. As $A(x_0)$
is uniformly bounded from below in $z$, it implies $\omega$ is lower
bounded in all time.

Finally, we prove that $(\omega(t),t)\in A(x_0)$ for $t\geq 0$ by
contradiction.

Suppose there exist $t>0$ such that $(\omega(t),t)\in\partial A(x_0)$,
and $(\omega(t+\delta),t+\delta)\not\in A(x_0)$. It means that
$\omega(t)=\gamma(t)$ and $\omega'(t)<\gamma'(t)$. On the other hand,
we compute
\[\gamma'(t)=\frac{m}{2}\phi'(L(X(t),t)+D)\frac{d}{dt}L(X(t),t)
\leq \frac{1}{2}|\phi'(L(x,t)+D)|V^\infty(0) m.\]
The last inequality is true as both $\partial(supp(\rho))$ and $X$ are
traveling with the speed between $u_{min}$ and $u_{max}$. It yields
$\frac{d}{dt}L(X(t),t)\leq V^\infty(t)\leq V^\infty(0)$.
Combined with the estimate on $\omega'$, we conclude that
$\omega'(t)\geq\gamma'(t)$ which leads to a contradiction.
\end{proof}

\subsection{Fast alignment property inside the vacuum}
In section \ref{sec:ff} we derived a much larger region of critical
threshold for $(x,t)\in supp(\rho)$, assuming fast
alignment property. In this subsection, we extend the enhanced result
to the vacuum area. We start by showing a fast alignment property where vacuum is 
involved. As the strength of viscosity at point $(\x,t)$ is determined
by $L(\x,t)$, it is natural to introduce the following definitions.

\begin{defn}[Level Sets]
For any level $\lambda\geq0$, define
\begin{align*}
\Omega^\lambda(t)&=\left\{X(t)~\left|~
\begin{cases}
\dot{X}(t)=\u(X,t)\\
X(0)=\x
\end{cases},~L(\x,0)\leq\lambda\right.\right\},\\
S^\lambda(t)&=\sup\left\{|\x-\y|,~~\x\in\Omega^\lambda(t),~
\y\in\Omega^0(t)\right\},\\
 V^\lambda(t)&=\sup\left\{|\u(\x,t)-\u(\y,t)|,~~\x\in\Omega^\lambda(t),~
\y\in\Omega^0(t)\right\}.
\end{align*}
\end{defn}
If $\lambda=0$, $\Omega^0(t)=supp(\rho(t))$, and $S^0(t), V^0(t)$
coincides with $\Vx(t), \Vu(t)$ respectively.
Moreover, $S^\lambda(0)=S_0+\lambda$.
If $\lambda=\infty$,
$V^\infty(t)$ coincide with the definition before.

\begin{thm}[Fast alignment on $\Omega^\lambda$]\label{thm:flockvacuum}
Let $(\rho, \u)$ be a global strong solution of system
\eqref{eq:main}. Suppose the influence function $\phi$ satisfies
$\displaystyle m\int_{\Vx_0}^\infty\phi(r)dr>V^\lambda(0)$.
Then,
there exists a finite number $D^\lambda$ (given by $D^\lambda= \psi^{-1}(V^\lambda(0)+\psi(S_0+\lambda), \ \psi(t):=m\int_0^t\phi(s)ds$), such that
$\displaystyle \sup_{t\geq 0}S^\lambda(t)\leq D^\lambda$ and $\displaystyle V^\lambda(t)\leq
V^\lambda(0) e^{-m\phi(D^\lambda)t}$.
\end{thm}

\begin{rem}
The proof of theorem \ref{thm:flockvacuum} follows the same idea in
proposition \ref{prop:SDDI} and theorem \ref{thm:flock} by
considering $X$ is a characteristic starting from
$\x\in\Omega^\lambda(0)$. We observe that $V^\lambda(t)$ still has an
exponential decay in time, with rate $m\phi(D^\lambda)$. When
$\lambda$ becomes larger, the rate becomes smaller. However, as long
as $\lambda$ is finite, we always have fast alignment.
\end{rem}

\noindent
We  are now ready to derive an 
improvement of lemma \ref{lem:vacuum}  using fast alignment
property.
\begin{proof}[Proof of theorem \ref{thm:vacuum}]
We repeat the proof of lemma \ref{lem:vacuum} using a better bound on
the term $Q$ which reads $|Q|\leq
V^{L(x_0,0)}(t)|\phi'(L(x,t))|m$. Also, we use a better bound on
$\frac{d}{dt}L(X(t),t)\leq V^{L(x_0,0)}(t)$. It yields the following
modified condition
\[V^{L(x_0,0)}(t)\leq\frac{m\phi^2(L(x,t)+D)}{4|\phi'(L(x,t))|+2\phi'(L(x,t)+D)|},\]
for all $x_0$ and $t$, with $(x,t)=(X(t),t)$ being a point on the
characteristics starting from $(x_0,0)$.

When $t=0$, let $\lambda=L(x_0,0)$, we get the condition
\eqref{eq:vacuumphi} stated in the theorem, i.e.
\[V^\lambda(0)\leq\frac{m\phi^2(\lambda+D)}{4|\phi'(\lambda)|+2\phi'(\lambda+D)|}.\]
Finally, we prove that if \eqref{eq:vacuumphi} holds, then the modified condition automatically
holds for all $t>0$. Take $\lambda=L(x,t)$; it suffices to prove that
$V^{L(x_0,0)}(t)\leq V^{\lambda}(0)$.
Applying theorem \ref{thm:flockvacuum}, we are left to prove
$V^{L(x_0,0)}(0)e^{-m\phi(D^{L(x_0,0)})t}\leq V^{L(x,t)}(0)$. This is
true if  $V^\lambda$ grows slower than exponential
rate in $\lambda$ which is the case for the finite $V^\infty$.
\end{proof}

\section{Critical thresholds for macroscopic Motsch-Tadmor model}\label{sec:MT}
In this section, we briefly discuss the critical thresholds phenomenon
of the MT system stated in theorem \ref{thm:mainMT}. The evolution of the gradient velocity matrix $M=\grad_\x\u$
reads
\begin{equation}\label{eq:MMT}
M'+M^2=\int_{\RR^n}\grad_\x
\left(\frac{\phi(|\x-\y|)}{\Phi(\x,t)}\right)\otimes(\u(\y,t)-\u(\x,t))\rho(\y,t)d\y -
M.
\end{equation}

The following proposition shows that \eqref{eq:MMT} admits a majorant system of the type \eqref{eq:Mdv}, with
\[
 \gamma= \Gamma=1,~~ \c=\frac{2\Lip{\phi}}{\phi(D)}, G=\phi(D),\]
 and the existence of critical thresholds for MT hydrodynamics follows along the lines of theorems \ref{thm:1DCS}, \ref{thm:lambdadv} in $n=1$ dimension and theorems \ref{thm:2DCS} and \ref{thm:2D} in $n=2$ dimensions. 
\begin{prop}
Suppose $(\rho,\u)$ is a solution of system \eqref{eq:MT}. Then,
for any $\x\in supp(\rho(t))$,
\[
\left|\int_{\RR^n}\partial_{x_j}\left(\frac{\phi(|\x-\y|)}{\Phi(\x,t)}\right)
(u_i(\y,t)-u_i(\x,t))\rho(\y,t)d\y\right|\leq\frac{2\Lip{\phi}}{\phi(D)}\Vu(t),
\quad i,j=1,\cdots,n.
\]
\end{prop}
\begin{proof} We begin with the estimate
\begin{align*}
&\left|\int_{\RR^n}\partial_{x_j}\left(\frac{\phi(|\x-\y|)}{\Phi(\x,t)}\right)
(u_i(\y,t)-u_i(\x,t))\rho(\y,t)d\y\right|\leq
\int_{\RR^n}|u_i(\y,t)-u_i(\x,t)|\left|\partial_{x_j}\left(\frac{\phi(|\x-\y|)}{\Phi(\x,t)}\right)
\rho(\y,t)\right|d\y\\
&\qquad\leq
\Vu(t)\int_{\RR^n}\frac{|\partial_{x_j}\phi(|\x-\y|)\rho(\y,t)\Phi(\x,t)
-\phi(|\x-\y|)\rho(\y,t) \partial_{x_j}\Phi(\x,t)|}{\Phi^2(\x,t)}d\y\\
&\qquad\leq \Vu(t)\left[\frac{\int_{\RR^n}|\partial_{x_j}\phi(|\x-\y|)|\rho(\y)d\y+
|\partial_{x_j}\Phi(\x,t)|}{\Phi(\x,t)}\right]
\leq \frac{2\Vu(t) \int_{\RR^n}|\partial_{x_j}\phi(|\x-\y|)|\rho(\y)d\y}{\Phi(\x,t)}.
\end{align*}
Since by theorem \ref{thm:flock}  $\Vx(t)\leq D$, then 
$\displaystyle \Phi(\x,t) =\int_{\RR^n}\phi(|\x-\y|)\rho(\y,t)d\y\geq\phi(D)\int_{\RR^n}\rho(\y,t)d\y$ 
for all $\x\in supp(\rho(t))$. Hence, 
$ \displaystyle 
\left|\int_{\RR^n}\partial_{x_j}\left(\frac{\phi(|\x-\y|)}{\Phi(\x,t)}\right)
(u_i(\y,t)-u_i(\x,t))\rho(\y,t)d\y\right|\leq
\frac{2\Lip{\phi}}{\phi(D)}\Vu(t)$.
\end{proof}
Using the same argument, we claim that theorem \ref{thm:mainMT} holds
with the following thresholds functions:
\begin{equation}\tag{Motsch-Tadmor: $\ThUL$}
\ThUL(0)=-1,~ \ThUL'(x)=
\begin{cases}
\displaystyle \frac{2\Lip{\phi}}{(1+\phi(D))\phi(D)},&x\to0+\\ \\
\displaystyle \frac{-\phi(D)\ThUL^2(x)-\phi(D)\ThUL(x)\mp 2\Lip{\phi}x}{-\phi^2(D)x}&x>0,\\
\end{cases}
\end{equation}

\noindent
{
\bf Funding statement}. Research was supported by NSF grants DMS10-08397, RNMS11-07444 (KI-Net) and ONR grant N00014-1210318.


\end{document}